\documentclass[11pt,reqno,twoside,a4paper]{article}
\usepackage{fullpage}
\usepackage{mysty}
\title{Inertial Bregman Proximal Gradient under Partial Smoothness}

\author{Jean-Jacques Godeme \thanks{INRIA Sophia  \& Université Côte d’Azur, CNRS, LJAD, France. \underline{e-mail:} jean-jacques.godeme@inria.fr}} 
\date{}
\begin{document}

\maketitle

\begin{abstract}
This work considers an Inertial version of  Bregman Proximal Gradient algorithm (IBPG) for minimizing the sum of two single-valued functions in finite dimension. We suppose that one of the functions is  proper, closed, and convex but non-necessarily smooth whilst the second is a  smooth enough function but not necessarily convex. For the latter, we ask the smooth adaptable property (smad)  with respect to some kernel or entropy which allows to remove the very popular global Lipschitz continuity requirement on  the gradient of the smooth part. We consider  the IBPG under the framework of the triangle scaling property (TSP) which is a geometrical property for which one can provably ensure acceleration for a certain subset of kernel/entropy functions in the convex setting.  Based on this property, we provide global convergence guarantees when the entropy is strongly convex under the framework of the Kurdyka-\L{}ojasiewicz (KL) property. Turning to the local convergence properties, we show that when the nonsmooth part is partly smooth relative to a smooth submanifold, IBPG has a finite activity identification property before entering a local linear convergence regime for which we establish a sharp estimate of the convergence rate. We report numerical simulations to illustrate our theoretical results on low complexity regularized phase retrieval. 
\end{abstract}
\section{Introduction}\label{sec:intro}
\subsection{Problem statement}
In this work, we study the following, non-necessarily smooth nor convex optimization problem
\beq\tag{$\calP$}\label{eq: generalpro}
\inf_{x\in\bbR^n} \left\{\Phi(x)\eqdef F(x)+G(x)\right\} ,
\eeq
under the following premises  
\begin{premise}{\ }\label{assump_A}
	\begin{enumerate}[label=(A.\arabic*)]
		\item \label{assump_A2} $F:\bbR^n\to\bbR$ is $C^1-$smooth,
		\item \label{assump_A1} $G:\bbR^n\to\overline{\bbR}$ is a  proper, lower semi-continuous and convex,
		\item \label{assump_A3} $\Phi$ is bounded from below, \ie $\inf \Phi > -\infty$ .
	\end{enumerate}
\end{premise}
In many applications like  machine learning, automatic differentiation and inverse problems in particular phase retrieval,  $G$ plays the role of a penalty or regularization term. It is intended to encode structural properties or prior  information about the set of solutions.  On the other hand, $F$ is  a data fidelity term between the true vector and the estimated vector. Throughout this paper,  $F$ can be nonconvex which allows us to cover a various problems of these fields. 
Before we describe the method that we propose to solve this problem, we recall the definition of Bregman divergence, an important notion that will be at the  heart in both our algorithm and his theoretical analysis.
\begin{definition}\label{Brgdf}\textbf{(Bregman divergence)} The  Bregman divergence associated with  a Bregman entropy $\phi$  is given by the following expression
	\begin{equation}\label{Bregmdef}
		D_{\phi}^v(x,y)\eqdef \left\{\begin{aligned} &\phi(x)-\phi(y)-\pscal{v;x-y}, &\si (x,y)\in \left(\dom\phi\times \inte\dom\phi\right),v\in\partial\phi(y)  ,\\
			&+\infty & \odwz. \end{aligned}\right.
	\end{equation}
\end{definition}
To solve the optimization problem \eqref{eq: generalpro}, we  link to  $\Phi$  a Bregman entropy function $\psi$ with respect to which $F$ is relatively smooth. In this work, we propose Algorithm~\ref{alg:BPGBT} which we coin Inertial Bregman Proximal Gradient (IBPG). 

\begin{algorithm}[htbp]
	\caption{Inertial Bregman Proximal Gradient}
	\label{alg:BPGBT}
	\textbf{Parameters:} $\kappa\in]1,2]$\;
	\textbf{Initialization:} $z_{-1}=x_{-1}, z_0=\xo\in\mathbb{R}^n$, $a_{-1}=a_0=1$,  and $0<\ua<\oa\leq1$\;
	\For{$k=0,1,\ldots$}{
		\vspace{-0.25cm}
		\begin{flalign} \tag{IBPG}\label{BPGalgo} 
			&\begin{aligned} 
				&\yk= \zk+ a_k(\xk-\zk); \\
				&\xkp = \Ppa{\nabla\psi+\gak\partial G}^{-1}\Ppa{\nabla\psi(\yk)-\gak\nabla F(\yk)}, \quad \gak=\frac{a_k^{\kappa-1}}{L}; \\
				&\zkp=\xk+a_k(\xkp-\xk); \\ 
				&\text{Choose }  a_{k+1}\in [\ua,\oa] \qobjq{s.t.} (a_{k+1}^{1-\kappa}+1)^{1/\kappa}(1-a_{k+1}) < a_k^{1/\kappa-1}/(1-a_k) . 
			\end{aligned}&
		\end{flalign}
	}
\end{algorithm}
It can be easily shown that $\yk$ can be written as a recursion of $(\xk,\xkm)$ with inertial parameters that depend on $(a_k,a_{k-1})$. When $\ak\equiv1$, we recover the Bregman Proximal Gradient (BPG without inertia) whose global convergence was already established in \cite{bolte_first_2017}. 

In \cite{hanzely_accelerated_2021}, a slightly different algorithm called Accelerated Bregman Proximal Gradient, was considered in the convex case. This scheme was inspired by the Improved Interior Gradient Algorithm for conic optimization as  developed in \cite[Section~5]{auslender_interior_2006}. It was shown in \cite{hanzely_accelerated_2021} that  Accelerated Bregman Proximal Gradient, indeed,  provides acceleration when the entropy $\psi$ satisfies the triangle scaling property (TSP) (see Definition~\ref{df:triangle-scaling}) with a convergence rate of $O(k^{-\kappa})$, where $ \kappa\in]1,2]$ is the triangle scaling exponent (TSE). For $\kappa=2$, one recovers the famous Nesterov-like accelerated rate. 

In the nonconvex case, another inertial scheme was analyzed in \cite{mukkamala_convex-concave_2020} when the Bregman entropy is also strongly convex and $G$ is weakly convex. Global convergence of the iterates under the KL property was proved there using line search on both the extrapolation (inertial) parameter and the descent step-size. Adding more structure of the  choice of the Bregman entropy, (\ie TSP) will allow to have a sharper analysis and results.
\subsection{Prior work}

\paragraph{Inertial Bregman proximal gradient algorithms}
Inertial methods emerged from the quest of accelerating the convergence of first-order optimization methods such as gradient descent. This starts with the Heavy-ball method with friction \cite{polyak_methods_1964}  for gradient descent ($G\equiv0$). This approach can be interpreted as a  discretization of a nonlinear second-order dynamical system, specifically an oscillator with viscous damping. This idea permeates now all the optimization techniques and has been applied for instance to the proximal point method \cite{alvarez_minimizing_2000, alvarez_inertial_2001} and to the inertial Forward-Backwards type methods \cite{moudafi_convergence_2003,attouch_dynamical_2014,lorenz_inertial_2015}. In terms of acceleration for convex programming, the accelerated FISTA method \cite{beck_fast_2009,nesterov_method_1983}  achieves a convergence rate of $O(k^{-2})$ for the sequence of objective functions which has been improved to $o(k^{-2})$ in \cite{attouch_rate_2016}, with convergent iterates in \cite{chambolle_convergence_2015}.  

For Bregman-based methods, first-order methods achieve the convergence $O(k^{-1})$ for the sequence of objective \cite{birnbaum2011distributed,bauschke_descent_2016,bolte_first_2017,lu2018relatively,Teboulle18} as in the Euclidean case. A natural question is whether the Bregman proximal gradient algorithm can be accelerated in the relative smooth framework. One can notice many attempts to answer this question within \cite{bauschke_descent_2016,lu2018relatively}, and the survey paper \cite[Section~6]{Teboulle18}. Positive answers have already been provided under  strict additional regularity premises. When the Bregman entropy $\psi$ is a strongly convex Legendre kernel and the smooth part of the objective has a Lipschitz continuous gradient, the Improved Interior Gradient Algorithm in \cite{auslender_interior_2006} admits an accelerated $O(1/k)$ convergence rate on the objective, by using the same inertial technique as Nesterov-type methods. For a subset of relatively smooth functions, \cite{hanzely_accelerated_2021} shows that the convergence rate of the objective can be improved from $O(1/k)$ to $O(1/k^{\kappa})$ where  $\kappa\in[1,2]$ is determined by some crucial triangle scaling property of the Bregman distance, whose genericity is unclear. The general case was still open until the work of \cite{dragomir_optimal_2022} which showed that the $O(1/k)$ rate is optimal for first-order algorithms over the subset of relatively smooth functions, and this cannot be improved in general. As far as the nonconvex case is concerned, inertial versions of the  proximal gradient method were analyzed in \cite{mukkamala_convex-concave_2020,wen_linear_2017,hien_inertial_2020,wu_inertial_2021} when the entropy is strongly convex. All these works use either backtracking or line search on both the extrapolation (inertial) parameter and the descent step-size. Global convergence of the iterates under KL were proved in \cite{mukkamala_convex-concave_2020,wu_inertial_2021} while \cite{wen_linear_2017} showed linear convergence under certain error bound condition.

\paragraph{Activity identification} Finite activity identification of underlying manifolds is an important phenomenon occurring when differents types of algorithms are used to solve structured nonsmooth minimization problems. Such analysis can be traced back to the work of \cite{burke_identification_1988} and \cite{dunn_convergence_1987}, where a nondegeneracy condition is used to ensure that the optimal active constraints are identified after finitely many iterations. They showed how polyhedral faces of convex sets could be identified finitely. Later \cite{wright_identifiable_1993} extended these results by introducing the concept of smooth identifiable surface and providing an algorithm that identifies the active constraints in convex problems in a finite number of iterations.  This work was then generalized in \cite{lewis_active_2002} to the notion of partial smoothness for general non necessarily convex problems. Among algorithms that identify active manifolds of partly smooth functions, one can cite  the (sub)gradient projection method, Newton-like methods, and the proximal point algorithm as shown in \cite{hare_identifying_2007,hare_identifying_2004,hare_identifying_2011} have shown for the (sub)gradient projection method, Newton-like methods, and the proximal point algorithm. A comprehensive study of finite activity identification as well as sharp local linear convergence has been established by Liang \etal in a series of papers for a variety of operator splitting algorithms: forward-backward-type algorithms including accelerated ones \cite{liang_local_2014,liang_activity_2017,liang_multi-step_2016}, for Douglas-Rashford/ADMM \cite{liang_activity_2015} and for the Primal-Dual splitting \cite{liang_local_2018}.

\paragraph{Strict saddle avoidance} A driving theme in nonconvex optimization, supported by empirical evidence, is that simple algorithms often work well in highly nonconvex and even nonsmooth settings by avoiding ``bad'' critical points. A growing body of literature provides one compelling explanation for this escape property or trap avoidance phenomenon. Namely, typical smooth objective functions provably satisfy the strict saddle property, meaning each critical point is either a local minimizer or has a direction of strictly negative curvature. For such functions, either randomly initialized gradient-type methods \cite{goudou_gradient_2009,panageas_gradient_2017}, or stochastically perturbed gradient methods \cite{pemantle_nonconvergence_1990,brandiere_algorithmes_1996} provably escape all strict saddle points, generically on initialization or on the noise. At the heart of the analysis, in all these work is the use of the Center Stable Manifold Theorem \cite[Theorem~III.7]{shub_global_1987} which finds its roots in the work of Poincaré. In \cite{goudou_gradient_2009}, it was proved that the heavy ball method with friction, applied to a $C^2-$Morse function, provably converges to a local minimizer generically with any initialization. Morse functions are known to be generic in the Baire sense in the space of $C^2-$functions (see \cite{aubin_applied_1984}). In \cite{lee_first-order_2019}, it has been shown that when the objective function is sufficiently smooth with a Lipschitz continuous gradient, a  broader class of first-order methods avoid strict saddle points. In the nonsmooth setting, Euclidean proximal methods as explored in \cite{davis_proximal_2022}, are  effective in avoiding a nonsmooth version of strict saddle points: ``active'' strict saddle points, which are strict saddle points with respect to an underlying activity manifold. In parallel to \cite{davis_proximal_2022}, we would like to mention the recent work of \cite{bianchi_stochastic_2023} which shows that stochastic subgradient descent also escapes ``active'' strict saddle points. The last two papers rely on important geometric conditions that turn out to be generic in the space of tame weakly convex functions.

\subsection{Contributions}
In this work, we study the global and local convergence properties of Algorithm~\ref{alg:BPGBT} for a subclass of Bregman kernels that satisfies the TSP property. The main contributions of this work are the following
\paragraph{Global convergence of the scheme} Under the KL-property, we show that bounded iterates generated by Algorithm~\ref{alg:BPGBT}  converge to a critical point of the objective function when the inertial parameters satisfy the condition that $\ak\in [\oa,\ua]\subset ]0,1[$. We also show that starting near an optimal solution, the generated sequence converges to it. 

\paragraph{Finite activity identification} We establish that sequence generated by IBPG enjoy a finite activity identification property. More precisely, we show that when the nonsmooth part $G$ is partly smooth with respect to an underlying manifold $\calM_{\xpa}$ near some critical point $\xpa$, the iterates will identify the manifold under an appropriate nondegeneracy condition. This identification property implies the existence of some large integer $K$ such that the iterates generated after this number lie in the manifold $\calM_{\xpa}$. The nondegeneracy condition cannot be relaxed in general as shown in \cite[Example~4.1, 4.2, 4.3]{hare_identifying_2007} respectively for the projected gradient descent, Newton method, and the proximal point algorithm.  

\paragraph{Local linear convergence analysis} After activity identification, we show that locally along the active manifold $\calM_{\xpa}$, we can linearize the iterates generated by IBPG when $F$ is locally $C^2$. In this case, we provide a spectral analysis of the linearized system, and under appropriate restricted injectivity, we exhibit a linear convergence regime for proper choice of the inertial parameter. This choice depends in particular on the TSE parameter $\kappa$, which generalizes the Euclidean case for which this parameter is just $2$. Our work hence extends that of \cite{liang_multi-step_2016} to non-euclidean, Bregman-based geometry.

\paragraph{Escape property in the smooth case} Equipped with the center stable manifold theorem, we study the trap avoidance of the inertial Bregman gradient method where $G\equiv0$. We show that the scheme generically avoids strict saddle points, which are critical points where the function has at least one direction of negative curvature.

Several numerical results are reported that confirm all our theoretical findings.

\subsection{Paper organization}
The rest of the paper is organized as follows. In section~\ref{sec:prelim}, we collect the notations and general definitions needed throughout this work. Section~\ref{sec:ibpg_convergence} contains the global convergence analysis of the  Inertial Bregman Proximal Gradient scheme with the Kurdyka-\L{}ojasiewicz property. Section~\ref{sec:ibpg_localconv} is a self-contain local analysis of the scheme that includes a finite activity identification of underline manifolds with the notion of partial smoothness, a local linearization of the iterates, and finally a local linear convergence. Section~\ref{sec:ibpg_trap_avoidance} contains a study of trap avoidance in the smooth case.
 In Section~\ref{sec:ibpg_numexp}, we highlight our results with a nonlinear inverse problems that arise in many areas of imaging science with applications that include diffraction imaging, astronomical imaging, and microscopy namely phase retrieval with prior information on the signal that we want to recover such as sparsity, low-rank, etc. All the technical proofs, the Bregman and  KL toolbox are postponed in the Appendices.

\section{Preliminaries}\label{sec:prelim}
\subsection{Notations}

In this work, for any nonempty convex set $\Omega\subseteq\bbR^n,$ the following subsets  $\inte(\Omega), \ri{\Omega},  \bd(\Omega),  \rbd{\Omega}$ and $\Aff{\Omega}$ denote respectively its interior, relative interior, boundary, relative boundary,  and its affine hull. We also denote  by $\proj{\Omega}$ the orthogonal projector onto $\Omega$. The subset parallel to the set $\Omega$  passing through the origin $0$ is denoted $ \LinHull(\Omega)=\Aff{\Omega}-\Omega$. We denote by $\pscal{\cdot,\cdot}$ the usual scalar product and  $\normm{\cdot}$ his  corresponding norm. $B(x,r)$ is the corresponding ball of radius $r$ centered at $x$ and $\mathbb{S}^{n-1}$ is the corresponding unit sphere. $\Gamma_0(\bbR^n)$ is the usual set of proper, closed, and convex functions. $\bbN$ is the set of non-negative integers. For $m \in \bbN^*, [m]$ is the discrete set $\{1,\cdots,m\}$. For any matrix $A$, $\transp{A}$ denote the transpose of $A$ and $\adj{A}$ the adjoint  (transpose conjugate) when $A$ is a complex matrix. We used the script $\lon$ for the partial order of Loewner between positive semidefinite matrices.  $\dom(f)$ is the domain of the function $f$. {$f^*$ denotes the Legendre-Fenchel conjugate of $f$.} 

\subsection{General Definitions}
Throughout this section, let $g:\bbR^n\to\overline{\bbR}=\bbR\cup\Ba{\infty}$ be a function such that $g\in\Gamma_0(\bbR^n)$.  

\begin{definition}\label{def:attentopo}{\ }\textbf{(Attentive topology)} 
        \begin{itemize}
        \item \textbf{(Attentive neighborhood)} A point $x\in\bbR^n$ is in the $g-$attentive neighborhood of  $\xpa\in\bbR^n$  , if  for all $r>0$, there exists $u\in ]0,r[$ and $\mu>0$ such that: 
        $\normm{x-\xpa}\leq u$ and $g(\xpa)<g(x)<g(\xpa)+\mu.$

        \item  \textbf{(Attentive convergence)} $x \xrightarrow{g}\xpa$ stand for $g-$attentive convergence \ie $x\rightarrow \xpa$ with $g(x)\rightarrow g(\xpa)$.
        \end{itemize}
        
\end{definition}

\begin{definition} \label{def:subdiff}{\ }\textbf{(Subdifferential)}
 The subdifferential of $g$ at a point $\xpa\in \bbR^n,$ is the set-valued operator defined as
\begin{equation}
\partial g(\xpa)\eqdef
    \bBa{v\in\bbR^n, g(x)\geq g(\xpa)-\pscal{v;x-\xpa} }.
    \end{equation}
\end{definition}
  
\begin{remark}\label{def: crit_point}\textbf{(Fermat's rule )} 
       A point $\xpa\in\bbR^n$ is a critical point of $g$ if $0\in\partial g(\xpa)$. We denote by $\crit{g}$ the set of all critical points of $g$.
 \end{remark}

\begin{definition}\label{def:Activ_mani}\textbf{(Active Manifold)}
     Let suppose that $g$ is a proper function and fix a submanifold  $\calM\subset\bbR^n$ containing a critical point $\xpa$ of $g$. Then, the set $\calM$ is called an active $C^p-$manifold $(p\geq2)$ around $\xpa$ if there exists a neighborhood $\calV_{\xpa}$ of the critical point $\xpa$ such that the following holds: 
     \begin{itemize}
         \item   The set $\calM\cap\calV_{\xpa}$ is a $C^{p}-$ smooth manifold  and the restriction of $g$ over $\calM\cap\calV_{\xpa}$ is also $C^{p}-$smooth.
        \item We have the lower bound:
         \begin{equation}\label{eq:Activ_mani}
           \inf\{\normm{v}:v\in\partial g(x), x\in\calV_{\xpa}\backslash\calM\}>0.  
         \end{equation}
     \end{itemize} 
    \end{definition}
\begin{remark}\label{rmk: Activ_mani}
This definition imposes an additional structural assumption to $g$.  Mainly, $g$ should be smooth along the manifold $\calM$ and the infimum of the norm of the subgradients are all bounded away from $0$.
\end{remark}
Now, we define active strict saddle points for nonsmooth functions. 
\begin{definition} \label{def: Active_stritsad}\textbf{(Active strict saddle)}
     We say that a point $\xpa$  is an active strict saddle point  of the nonsmooth function $g$ if 
    \begin{enumerate} [label=(\roman*)]
        \item\label{def: Active_stritsad-i}$\xpa\in\crit{g}$ \ie,  $0\in\partial g(\xpa)$.
        \item\label{def: Active_stritsad-ii} There exists an active manifold $\calM$ to the point $\xpa$ according the Definition~\ref{def:Activ_mani}.  
        \item\label{def: Active_stritsad-iii}  $g$ decreases quadratically along some direction $v\in\tgtManif{\calM}{\xpa}$ \ie there exists some smooth curve $\tau\subset\calM$ with $\tau(0)=0$ and $\dot{\tau}(0)=v$ such that: 
    \[
    \left.\frac{d^2}{dt^2}\pa{g\circ\tau}(t)\right|_{t=0}<0. 
    \]       
    \end{enumerate}
    \end{definition}
    Let us denote by $\actstrisad{g}$ the set of all active strict saddle points of $g$. 
\begin{remark}\label{rmk: activ_saddle}{\ }
\begin{itemize}
\item If every critical point of $g$ is either an active strict saddle point or a local/global minimizer thus $g$ has the strict saddle property. 
\item Active strict saddle generalizes the notion of strict saddles for $C^2-$smooth functions to nonsmooth functions. Indeed when $g$ is $C^2-$smooth we have $\calM=\bbR^n$ and thus active strict saddles are exactly strict saddles points and we denote the set $\strisad{g}$. 
\item  Condition \ref{def: Active_stritsad-iii} precisely mean that there exists $v\in\tgtManif{\calM}{\xpa}$ such that we have $d^2f_{\calM}(\xpa)(v)<0$, where $d^2f_{\calM}(\cdot)(\cdot)$ denotes the second order  subderivative along the manifold $\calM$ (for more details see \cite[Chapter 13, Section B]{rockafellar_variational_1998} ).    
\item This notion of an active strict saddle point is generic for functions that are closed, weakly convex, and semi-algebraic. 
\end{itemize}
\end{remark}

\section{Global Convergence Analysis}\label{sec:ibpg_convergence}
\subsection{Main assumptions}
We will need the following premises on $\psi$, which will be invoked jointly or separately in our proofs.
\begin{premise}\label{assump_B}{\ }
	\begin{enumerate}[label=(B.\arabic*)]
		\item \label{assump_B1}  $\psi$ is a $C^2-$smooth $\sigma_{\psi}$-strongly convex function with $\sigma_{\psi} > 0$,
		\item \label{assump_B4} $F$ is $L-$smooth relative to $\psi$ on $\bbR^n$,
		\item \label{assump_B3} $\nabla F$ and $\nabla \psi$ are Lipschitz continuous on bounded subsets of $\bbR^n$.
	\end{enumerate}
\end{premise}
\begin{remark} {\ }
	\begin{itemize}
		\item Strong convexity of $\psi$ plays an important role to establish global convergence of the sequences of iterates,
		\item $C^2-$smoothness of $\psi$ is only needed occasionally and some of our statements remain true even without it,
		\item Premise~\ref{assump_B1} implies that $\psi$ is Legendre and thus $\nabla\psi$ is a bijection on $\bbR^n$ whose inverse is $\nabla \psi^*$ (see Remark~\ref{rmk: Legend_func}). Moreover, strong convexity implies that $\nabla^2\psi^*(\nabla\psi(x)) = \nabla^2\psi(x)^{-1}$; see \eg \cite[Lemma~2.2]{laude_bregman_2020}.  
		\item Premise~\ref{assump_A1} together with \ref{assump_B1} imply that the $D$-prox operator of index $\gamma > 0$
		\[
		\Ppa{\nabla\psi+\gamma\partial G}^{-1} \circ \nabla\psi: x \in \bbR^n \mapsto \Argmin_{z \in \bbR^n} G(z) + \frac{1}{\gamma} D_\psi(z,x)
		\]
		is single-valued; see \cite[Proposition~3.22]{bauschke_bregman_2003}. Strong convexity of $\psi$ can be weakened to strict convexity and legenderness if $\psi+\gamma G$ is supercoercive. If convexity of $G$ is removed, the $D$-prox is nonempty and compact-valued if $\psi$ is Legendre and $\psi+\gamma G$ is supercoercive; see \cite[Lemma~3.1]{bolte_first_2017}. 
		\item Our analysis and results can be extended to handle the constrained case where \eqref{eq: generalpro} is solved over a closed convex set. This necessitates that $\psi$ to be a barrier function of the constraint set and some technical (domain) adaptations of our assumptions that we prefer to avoid here for the sake of clarity\footnote{Anyway, our focus being on phase retrieval, our current setting is sufficient.}.
	\end{itemize}
\end{remark}

\subsection{Convergence analysis}\label{sec:ibpg_globconv}
Our main result states that if $\Phi$ is also a KL function, then bounded iterates of IBPG converge to a critical point of $\Phi$. 
\begin{theorem}\label{thm:Global_conv_ABPG}
	Consider problem \eqref{eq: generalpro} under the Premises~\ref{assump_A} and  \ref{assump_B}. Suppose that the sequence of IBPG parameters $\seq{\ak}$ are chosen as in Algorithm~\ref{alg:BPGBT}. Then,
	\begin{enumerate}[label=(\roman*)]
		\item  \label{thm:Global_conv_ABPG-iii}
		$\sum_{k \in \bbN}\normm{\xk-\xkm}^2<\infty$ and
		\[
		\min\limits_{0 \leq i\leq k} \normm{x_i-x_{i-1}}^2\leq \frac{2(\sigma_{\psi}\nu L)^{-1}\Psi_0(x_0,x_{-1})}{k+1} .
		\]
		
		\medskip
		
		Assume moreover that $\Phi$ and $\psi$ satisfy the KL property, that $a_k \equiv a\in[\ua,\oa]$ for all $k \geq K$, where $K$ is arbitrarily large. If the sequence of IBPG iterates $\seq{\xk}$ is bounded then,
		\item \label{thm:Global_conv_ABPG-i} all sequences $\seq{\xk}$, $\seq{\yk}$ and $\seq{\zk}$ have finite length and converge to the same limit in $\crit{\Phi}$.
		\item \label{thm:Global_conv_ABPG-ii} If $\Argmin(\Phi) \neq \emptyset$ and IBPG is started near a global minimizer $\xsol$ in the $\Phi-$attentive topology, then the generated sequences converge to $\xsol$.
	\end{enumerate}
\end{theorem}
We defer the proof to Section~\ref{pr:thm:Global_conv_ABPG}. 
\begin{remark}\label{rmk: global_conv}{\ }
	\begin{itemize}
		\item The fact that the sequences are  bounded, is a standard premise for the global convergence of the sequence in the nonconvex case. Coercivity of $\Phi$ is for instance sufficient to ensure it.  
		
		\item Choosing $a_k$ constant for $k$ large enough is a standard strategy for the Lyapunov analysis in the nonconvex case. A similar strategy is also used in \cite{bolte_nonconvex_2018} and \cite{mukkamala_convex-concave_2020}. This also makes sense in practice, a fixed inertial parameter as in the heavy ball method with friction is popular for inertial algorithms. 
	\end{itemize} 
\end{remark}

The choice of $\seq{\gamma_k}$ and $\seq{\ak}$ devised in Algorithm~\ref{alg:BPGBT} is sufficient for our analysis. Actually, we only need that 
\[
\gamma_k \in ]0,1/L] \qandq (L+\gamma_k^{-1})(1-a_k)^{\kappa}(1-a_{k-1})^{\kappa} < \gamma_{k-1}^{-1}.
\]
For instance, if $\gamma_k \equiv 1/L$, then it is sufficient that $(1-a_k) < 2^{-1/\kappa}/(1-a_{k-1})$ which is easy to verify since $a_k \in ]0,1]$. 

There are many possible choices of the sequence $\seq{a_k}$ that obey the condition of Algorithm~\ref{alg:BPGBT}. For instance, take $a_k = \frac{k+1}{k+1+\alpha} \in [1/(1+\alpha),1]$, where $\alpha > 0$. To verify that the condition holds, observe that $(a_{k+1}^{1-\kappa}+1)(1-a_{k+1})^{\kappa}$ is decreasing in $k$. On the other hand, the function $h: [1/(1+\alpha),1] \mapsto a^{1-\kappa}/(1-a)^{\kappa}$ has a unique minimum on $[0,1]$ at $a_{\min} = \max\Ppa{1/(1+\alpha),(\kappa-1)/(2\kappa-1)}$. Thus, for the inequality to hold true, it is sufficient that $(a_{1}^{1-\kappa}+1)(1-a_{1})^{\kappa} < h(a_{\min})$ for all $\kappa \in ]1,2]$. This is achieved by taking $\alpha \leq 3$.

\section{Local Convergence Analysis}\label{sec:ibpg_localconv}
In this section, we present the local analysis of the Inertial Bregman Proximal Gradient. 
We start with the following definition. 
\begin{definition}\textbf{(Nondegenerate critical point for composite function)}
	We say that a critical point satisfies the nondegeneracy condition for the composite function $\Phi$ if: 
	\begin{equation}\tag{ND}\label{eq: Nondegen}
		-\nabla F(\xpa)\in\ri{\partial G(\xpa)}. 
	\end{equation}
	Let $\nd{\Phi}$ denote the set of critical points satisfying this condition for $\Phi$.  
\end{definition}
\begin{remark}\label{}{\ }
	\begin{enumerate}
		\item[(i)] Applying \cite[Proposition~10.12]{drusvyatskiy_optimality_2014}, it turns out that for $G\in\Gamma_0(\bbR^n)$, the notion of identifiable manifold is equivalent to partial smoothness around an active smooth manifold combined with the nondegeneracy condition. 
		\item[(ii)] Suppose that $\Phi$ is a semi-algebraic function, and more generally a function definable on an o-minimal structure. It follows from \cite[Theorem~4.16]{drusvyatskiy_generic_2016} that generically on $v \in \bbR^n$, the function $\Phi_v\eqdef \Phi(x)-\pscal{v,x}$ has a finite number of critical points and each critical point is nondegenerate and admits an identifiable manifold. In plain words, assumption \eqref{eq: Nondegen} is generic.
	\end{enumerate}
\end{remark}

\subsection{Finite activity identification of IBPG}
The following result shows that IBPG generates a sequence that identifies active manifolds in finite time.
\begin{lemma}[Finite time activity identification]\label{lem:LemFinite}
	Let us consider an instance of Algorithm~\ref{alg:BPGBT} such that $\seq{\xk}$ is bounded. Let $\xpa\in\crit{\Phi}$ be the limit of the sequence and suppose that $G\in\PSF{\xpa}{\calM_{\xpa}}$ with $\xpa\in\nd{\Phi}$.  Under the same Premises as Theorem~\ref{thm:Global_conv_ABPG}, there exists a constant $K$ large enough such that for all $k\geq K, \xk\in\calM_{\xpa}.$
	\begin{enumerate}[label=(\roman*)]
		\item $\calM_{\xpa}$ is an affine subspace, then $\calM_{\xpa}=\xpa+T_{\xpa}$ and $\seq{\yk},\seq{\zk}\in\calM_{\xpa},k\in K,$
		\item If moreover, $G$ is locally polyhedral around $\xpa$ then for all $k\geq K,$  the remaining sequences satisfy $\yk,\zk\in\calM_{\xpa}$ and $\nabla^2_{\calM_{\xpa}}G(\xk)=0.$ 
	\end{enumerate}
\end{lemma}
See the Section~\ref{proof_finite} for the proof of this lemma. 

\subsection{Local linearization of IBPG}
In this section, let us consider $\xpa$ a critical point of $\Phi$ and let $\calM_{\xpa}$ be a $C^2-$smooth submanifold such that $G\in\PSF{\xpa}{\calM_{\xpa}}$. Let us assume that the smooth part $F$ is $C^2$ around $\xpa$. Let us denote $T_{\xpa}\eqdef\tgtManif{\calM}{\xpa}$ and fix a stepsize $\gak \in ]0,1/L]$. For the rest of the analysis, we define the following matrices which help us to capture the local behavior of the iterates.  
\begin{align}\label{eq:matrices-analysis}
	&H_{F}\eqdef\gamma\proj{T_{\xpa}}\nabla^2F(\xpa)\proj{T_{\xpa}},\quad H_{\psi}\eqdef\proj{T_{\xpa}}\nabla^2\psi(\xpa)\proj{T_{\xpa}},\quad V\eqdef H_{\psi}-H_{F},\\
	&U\eqdef\gamma\nabla^2_{\calM_{\xpa}}\Phi(\xpa)\proj{T_{\xpa}}-H_F.\nonumber
\end{align}
where $\nabla^2_{\calM_{\xpa}}\Phi$ denotes the  Riemannian Hessian of $\Phi$ along the submanifold $\calM_{\xpa}$ and $\proj{T_{\xpa}}$ the projection onto $T_{\xpa}$. 
\begin{remark}\label{rmk:positive-U}$~$
	Since $G\in\Gamma_0(\bbR^n)$, \cite[Lemma~4.3]{liang_activity_2017} shows that $U$ is symmetric positive semi-definite under the condition that $\xpa\in\nd{\Phi}$ or that $\calM_{\xpa}$ is an affine space. Therefore $H_{\psi}+U$ is symmetric positive definite and hence invertible. Let us denote 
	\[
	W \eqdef \Ppa{H_{\psi}+U}^{-1} . 
	\]We have that $W$ is symmetric positive with eigenvalues in $\left]0,1/\sigma_{\psi}\right].$
\end{remark}
\begin{definition}\textbf{(Restricted injectivity)} We say that a critical point $\xpa$ satisfies the restricted injectivity condition if there exists $\sigma\geq 0$  such that  the following condition holds true 
	\begin{equation}\label{GRI}
		\forall h\in T_{\xpa},\quad \pscal{h,\Ppa{\nabla^2F(\xpa)-\sigma\nabla^2\psi(\xpa)}h}\geq 0.               
	\end{equation}
\end{definition}
We denote by $\gri{\Phi}$ the set of all critical points where the restricted injectivity is satisfied. In this case, the local continuity of the Hessian of $F$  implies that $\ker{(\nabla^2F(\xpa))}\cap T_{\xpa}=\{0\}$. Observe that by \cite[Theorem~4.16]{drusvyatskiy_generic_2016} and strong convexity of $\psi$, \eqref{GRI} is equivalent to the fact that $\xpa$ is a stable strong local minimizer of $\Phi$.

Let $a\in[\ua,\oa]$ and define $\rk\eqdef \xk-\xsol$ and $\dk\eqdef\begin{pmatrix} \rk \\ \rkm \end{pmatrix}$, we will need the following key matrix
\begin{equation}\label{eq:matrix-linearization}
	M\eqdef\begin{bmatrix}
		(2a-a^2)WV&(1-a)^2WV\\
		\Id&0     
	\end{bmatrix}.
\end{equation}
At this step, we can describe the local behavior of the sequence generated by Algorithm~\ref{alg:BPGBT}. The next result is a local linearization of the iterative scheme.   
\begin{proposition}\label{pro:local-linearization}
	Consider the problem \eqref{eq: generalpro} under the Premises \ref{assump_A} and \ref{assump_B}. Let suppose that the sequences produced by Algorithm~\ref{alg:BPGBT} 
	converge to  $\xpa\in\nd{\Psi}\cap\gri{\Psi}$ with $G\in\PSF{\xpa}{\calM_{\xpa}}$. If $F$ is $C^2$ locally around $\xpa$ and the inertial parameter sequence $\seq{\ak}$ satisfies $\ak \to a$ then for $k$ large enough, we have
	\begin{equation}\label{eq:local-linearized}
		d_{k+1}=Md_k+o(\normm{d_k}). 
	\end{equation}
	The little $``o"$ term disappears when $G$ is locally polyhedral and $a_k$ is chosen constant.
\end{proposition}
 Section~\ref{proof:local-linearization}  contains  the proof of this proposition. 
\begin{remark}\label{rmk: local_linearization}{\ }
	\begin{enumerate}
		\item[(i)] If \(a_k\equiv 1\), we recover the  Bregman Proximal Gradient and we have the following
		linearized iteration
		\[r_{k+1} = WVr_k + o(\|r_k\|).\]
		\item[(ii)] When the kernel is the energy, \ie $\psi = \normm{\cdot}^2/2$, a similar analysis has been done for a symmetric version of the inertial Forward-Backward \cite[Proposition~4.5]{liang_activity_2017} with a different choice of inertial parameters. Our result is however different and it involves the new matrix $H_{\psi}$ which makes the spectral analysis of $M$ more intricate.
	\end{enumerate}
\end{remark}

\subsection{Spectral properties of $M$}
Our goal now is to show local linear convergence of IBPG. Towards this, we examine the structure of the locally linearized iteration given in \eqref{eq:local-linearized}.  Our goal is to upper-bound (strictly) the spectral radius of $M$  by $1$ and subsequently draw conclusions using standard reasoning.
We will relate the eigenvalues of $M$ to those of $WV$. Let $\eta$ and $\varrho$ be an eigenvalue of $WV$ and $M$ respectively. We denote $\etam$ and $\etaM$ as the smallest and largest (signed) eigenvalues of $WV$, and $\rho(M)$ as the spectral radius of $M$. When $G\) is a general partly smooth function, then $U$ is nontrivial, we have the following proposition. 
\begin{proposition}\label{pro:poly-spect-anal}
	Let us define 
	$\Lambda\eqdef \left|q_\psi(\xpa)-\gamma\sigma\right|$ where $q_{\psi}(\xpa)=\frac{\lambda_{\max}(\nabla^2\psi(\xpa))}{\sigma_{\psi}}$. Denote $q_F(\xpa) = \frac{L}{\sigma}$.
	
	\begin{enumerate}[label=(\roman*)]
		\item  \label{pro:poly-spect-anal-1}
		Let  $\begin{pmatrix} r_1 \\ r_2\end{pmatrix}$ be an eigenvector of \(M\) corresponding to an eigenvalue \(\varrho\) then it must satisfy $r_1= \varrho r_2.$
		Besides, 
		\(r_2\) is an eigenvector of \(WV\) associated with the eigenvalue \(\eta\), where \(\eta\) and \(\varrho\) satisfy the relation
		\begin{equation}\label{eq:vrho-eta-1}
			\varrho^2-(2a-a^2)\varrho\eta-(1-a)^2\eta=0,
		\end{equation}
		and $\rho(M)\leq\rho(WV)<\Lambda$ if, and only if, 
		\[
		\frac{1}{2a^2-4a+1}<\etam.
		\] 
		
		\item \label{pro:poly-spect-anal-ii}
		If the inertial parameters are chosen such that $a\in[\ua,\oa]$ with $\ua>\sqrt[\kappa-1]{q_F(\xpa)\pa{q_\psi(\xpa)-1}}$ and $\oa<\sqrt[\kappa-1]{q_F(\xpa)\pa{q_\psi(\xpa)+1}}$ then $\Lambda<1$. 
	\end{enumerate}   
\end{proposition}
See section~\ref{pr:pro:poly-spect-anal} for the proof.  
\begin{remark}\label{rm:poly-spect-anal}$~$
	\begin{itemize}
		\item[(ii)] \label{rmk:spectral-matrice-ii} In the euclidean case, $\Lambda=\Ppa{1-\ak^{\kappa-1}\frac{\sigma}{L}} < 1$ since $\sigma \leq L$ and $\ak\in]0,1]$.
		
		\item[(iii)] \label{rmk:spectral-matrice-iii}  Claim~\ref{pro:poly-spect-anal-ii} states that $\rho(M) < 1$ whenever $a$ is small enough while being bounded away from zero. 
	\end{itemize}  	
\end{remark}

\subsection{Local linear convergence}
We can now state the local linear convergence result.
\begin{theorem}\label{thm:local-linear}\textbf{(Local linear convergence)}
	Consider the problem \eqref{eq: generalpro} under Assumptions~\ref{assump_A}-\ref{assump_B}. Let $\seq{\xk}$ be a bounded sequence produced by Algorithm~\ref{alg:BPGBT} that converges to $\xpa\in\nd{\Phi}\cap\gri{\Phi}$ with $G\in\PSF{\xpa}{\calM_{\xpa}}$, and assume that $F$ is $C^2$ locally around $\xpa$. If $a$ is such that Proposition~\ref{pro:poly-spect-anal}\ref{pro:poly-spect-anal-ii} holds, then $\seq{\xk}$ converges locally linearly to $\xpa$. More precisely, given any $\rho\in\left[\rho(M),1\right[$, there exist $K\in\bbN$ large enough such that $\forall k\geq K$ ,
	\begin{equation}
		\frac{\normm{\zk-\xpa}}{\normm{z_K-\xpa}}=O(\rho^{k-K}). 
	\end{equation} 
\end{theorem}
\begin{proof}
	First use the global convergence result combined with the local linearization Proposition~\ref{pro:local-linearization} and the spectral analysis of $M$ in Proposition~\ref{pro:poly-spect-anal}. Then conclude by standard arguments.
\end{proof}

\section{Escape Property in the Smooth Case}\label{sec:ibpg_trap_avoidance}
\subsection{Trap avoidance for the inertial mirror descent}
Throughout this subsection, we assume that $G\equiv0$. Thus IBPG reduces to the Inertial Mirror Descent (IMD)  which is a variant of the Improved Interior Gradient Algorithm \cite{auslender_interior_2006}. We assume that the algorithm is run with a fixed stepsize and a fix inertial parameter. The scheme is summarized in Algorithm~\ref{alg:IMD} for the reader's convenience. 
\begin{algorithm}[htbp]
	\caption{Inertial Mirror Descent}
	\label{alg:IMD}
	\textbf{Parameters:} $\kappa\in]1,2]$\;
	\textbf{Initialization:} $z_{-1}=x_{-1}, z_0=\xo\in\mathbb{R}^n$, $a_{-1}=a_0=1$,  and fix $a\in]0,1]$\;
	\For{$k=0,1,\ldots$}{
		\vspace{-0.25cm}
		\begin{flalign*} 
			&\begin{aligned} 
				&\yk= \zk+ a(\xk-\zk); \\
				&\xkp = \nabla\psi^{-1}\Ppa{\nabla\psi(\yk)-\gamma\nabla F(\yk)}, \quad \gamma=\frac{a^{\kappa-1}}{L}; \\
				&\zkp=\xk+a(\xkp-\xk). \\  
			\end{aligned}&
		\end{flalign*}
	}
\end{algorithm}

\begin{theorem}\label{thm: iga_glbl_escp}\textbf{(Trap avoidance of IMD.)} 
	Consider the minimization problem \eqref{eq: generalpro} with $G\equiv0$ under Assumption~\ref{assump_A}-\ref{assump_B} and let $\xpa\in\crit{\Phi}$. Then for almost all initilizers $(x_0,x_{-1})$ of IMD, the generated sequences converge to a critical point that is not a strict saddle point. 
\end{theorem}
We defer the proof to Section~\ref{pr:thm: iga_glbl_escp}. Clearly, this means that if $(x_0,x_{-1})$ is drawn at random from a distribution with has a density \wrt Lebesgue measure, then with probability one, IMD converges to a critical point which is not a strict saddle.

\subsection{Challenges of the escape property for IBPG in the nonsmooth case}
In this section, we discuss the difficulties and challenges posed by the case where $G\neq0$ in \ref{eq: generalpro} is nonsmooth. First, one has to adapt the notion of strict saddles to the nonsmooth setting. Adopting the terminology in \cite{davis_proximal_2022}, we introduce the following notion of active strict saddles.

\begin{definition}[Active strict saddle] \label{def: Active_stritsad}
	Let us consider  $g:\bbR^n\to\overline{\bbR}$. We say that a point $\xpa$  is an active strict saddle point  of the nonsmooth function $g$ if 
	\begin{enumerate} [label=(\roman*)]
		\item\label{def: Active_stritsad-i}$\xpa\in\crit{g}$ \ie,  $0\in\partial g(\xpa)$.
		\item\label{def: Active_stritsad-ii} There exists an active manifold $\calM$ at the point $\xpa$.   
		\item\label{def: Active_stritsad-iii} The Riemannian Hessian of $g$ at $\xpa$ has a at least one negative eigenvalue. 
	\end{enumerate}
	Let us denote by $\actstrisad{g}$ the set of all active strict saddle points of $g$. 
\end{definition}

In the Euclidean setting, the authors in \cite{davis_proximal_2022} showed that proximal methods with weakly convex and definable functions generically avoid active strict saddle points. The core of their proofs is again the center stable manifold theorem. In turn, the regularity required by this theorem heavily rely on the properties of the proximal mapping in the euclidean setting, and in particular its firm nonexpansiveness, as well as Lipschitz continuity of the gradient of the smooth part.


In the Bregman setting, these properties are not true anymore. In fact, the proof strategy consists in characterizing the regularity and the spectrum of the following fixed point mapping that characterizes Algorithm~\ref{alg:BPGBT},  
\[
\mathbf{T}\pa{x_2,x_1}=\begin{bmatrix}
	\pa{\nabla\psi+\gamma\partial G}^{*}\Ppa{\nabla\psi\Pa{y(x_2,x_1)}-\gamma\nabla F(y(x_2,x_1))}\\
	x_2
\end{bmatrix},
\]
where
\[
a\in[\ua,\oa] \qandq y(x_2,x_1)=\pa{2a-a^2}x_2+ \pa{1-a}^2x_1.
\]
One has that $\xpa\in\crit{\Phi}$ if and only if $(\xpa,\xpa)$ is a fixed point of the operator $\mathbf{T}$. We have the following intermediate result.
\begin{lemma}\label{lem:unstablebpg} 
	Consider the minimization problem \eqref{eq: generalpro} under Premises~\ref{assump_A}-\ref{assump_B} and let $\xpa\in\crit{\Phi}$. Then $\mathbf{T}$ is a $C^1-$smooth function in a neighborhood of $(\xpa,\xpa)$. Besides, if $\xpa$ is an active strict saddle of $\Phi$ then the Jacobian $D \mathbf{T}(\xpa,\xpa)$ has a real eigenvalue that is strictly greater than one. 
\end{lemma}
We defer the proof to Section~\ref{pr:lem:unstablebpg}. 

\begin{remark}\label{rmk:escape-bregm}{\ }
	
This Lemma extends \cite[Theorem~4.1]{davis_proximal_2022} to the inertial Bregman proximal gradient setting. The key insight is that, in the vicinity of the critical point denoted by $\xpa$, any optimization problem can be locally approximated on the active manifold. This allows us to circumvent the nonsmoothness issues present in the general case.

However, extending \cite[Theorem~4.1]{davis_proximal_2022} to the Bregman setting presents a significant challenge. In the Euclidean framework, the arguments rely heavily on \cite[Corollary~2.12]{davis_proximal_2022} of the center manifold theorem. This result hinges on the hypothesis that the mapping $A$ is a global lipeomorphism. Unfortunately, this is not true in the Bregman setting, as $A$ is only a local lipeomorphism but not a global one. This suggests that a different proof strategy is needed which calls for future work as we will discuss in our perspectives. 
\end{remark}

\section{Numerical Experiments}\label{sec:ibpg_numexp}
In this section, we discuss some numerical experiments to illustrate our theoretical results.
\subsection{Phase retrieval}
We apply our results to regularized phase retrieval. Recall that the goal is to recover a vector $\avx\in\bbR^n$ from quadratic measurements 
\[y=|A\avx|^2 \in \bbR^m,\]
where $A:\bbR^n\rightarrow\bbR^m$ is a linear operator. We can reformulate this problem as an optimization problem in the form
\begin{equation}\label{eq:phase_retrieval}
	\min\limits_{x\in\bbR^n}\Phi(x)=\frac{1}{4m}\normm{y-|Ax|^2}^2+\lambda R(x), \quad \lambda>0.
\end{equation}
where $R \in \Gamma_0(\bbR^n)$. $R$ is a regularizer that that promotes objects sharing a structure similar to that of $\avx$. Problem~\eqref{eq:phase_retrieval} is an instance of \eqref{eq: generalpro} with $F(x)=\frac{1}{4m}\normm{y-|Ax|^2}^2$ and $G(x)=\lambda R(x)$.

Let us observe that $F$ is a semi-algebraic data fidelity term and that $F\in C^2(\bbR^n)$ but is nonconvex (though weakly convex). Besides, $\nabla F$ is not Lipschitz continuous. Therefore, we associate to $F$ the  kernel function 
\begin{equation}\label{eq:entropy-n}
	\psi(x)=\frac{1}{4}\normm{x}^4+\frac{1}{2}\normm{x}^2. 
\end{equation}
$\psi\in C^2(\bbR^n)$ is full domain and $1-$strongly convex function with a gradient that is Lipschitz over bounded subsets of $\bbR^n$. We have that $F$ is smooth relative to $\psi$ (see \cite[Lemma~2.8]{godeme2023provable}). Thus all our Premises~\ref{assump_A}-\ref{assump_B} are fulfilled.

\subsection{Experiments setup}
 This experimental set-up has been inspired by the work \cite{liang_multi-step_2016,liang_activity_2017}. Throughout our experiments, $A$ is drawn from the standard Gaussian ensemble, \ie, the entries of $A$ are \iid mean-zero and standard Gaussian. We solve this problem using Algorithm~\ref{alg:BPGBT} with $\ak\equiv1$. For each numerical experiment, we run the algorithm with a constant step-size $\gamma=\frac{0.99}{3+10^{-4}}$. For $R$, we have tested several regularizers as described hereafter. 
\begin{example}\textbf{(\onenorm).} For any  $x\in\bbR^n$, the \onenorm is given by 
	$R(x)=\norm{x}{1}\eqdef\sum_{i=1}^n|x_i|,$ which is partly smooth at any $x$ relative to the linear subspace
	\[
	\calM=T_x\eqdef\left\{u\in\bbR^n:\supp(u)\subset\supp(x)\right\},\quad\supp(x)\eqdef\{i:x_i\neq0\}.
	\] 
	
	The underlying vector $\avx$ is taken to be sparse with $s=12$ non-zeros entries for a vector of size $n=128$. The number of quadratic measurements is taken as $m=0.5\times s^{1.5}\times\log(n)$. As there is no noise, we took $\lambda=10^{-5}$. Figure~\ref{fig: solving_l1} shows the recovery results. The left plot of Figure~\ref{fig: solving_l1} displays the relative error of the iterates vs the number of iterations. On the right plot, we display the support of the iterates. Clearly, the left plot shows that Algorithm~\ref{alg:BPGBT} identifies the correct support after 300 iterations and converges to the true vector. The left plot confirms what is anticipated by our analysis, that the relative error converges locally linearly (see the dashed line). The local linear convergence rate is in very good  agreement with the one we predicted.
	\begin{figure}[htbp]
		\centering
		\includegraphics[trim={0cm 8cm 0 8cm},clip,width=0.7\linewidth]{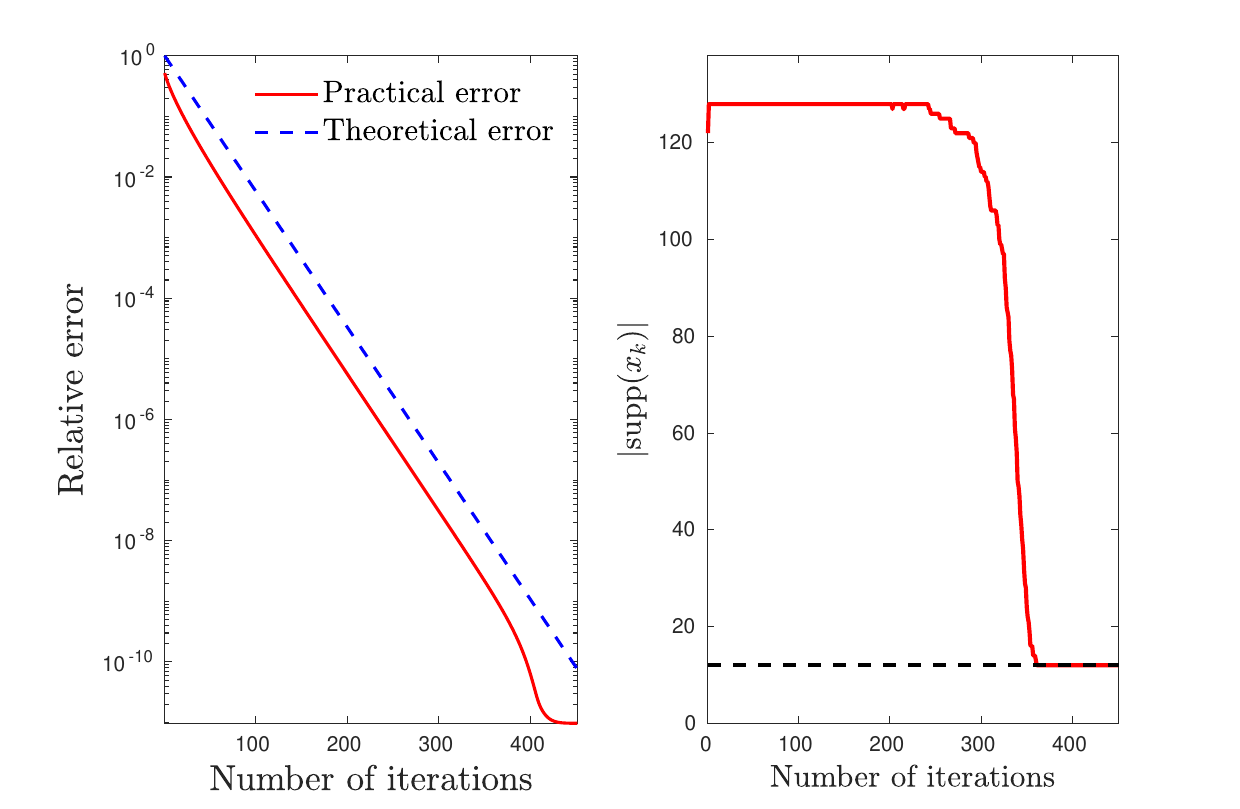}
		\caption{Phase retrieval by solving \eqref{eq:phase_retrieval} with the $\ell_1-$norm regularizer.}
		\label{fig: solving_l1}
	\end{figure}
\end{example}

\begin{example}\textbf{(\onetwonorm).} Here, we take $R$ as the group/block Lasso which si designed to promote group sparsity. Let $\{1,\cdots,n\}$ be partitioned into nonoverlapping blocks $\calB$ such that: $\bigcup\limits_{b\in\calB}=\{1,\cdots,n\}$. The \onetwonorm  of $x$ is given by $R(x)=\norm{x}{1,2}\eqdef\sum_{b\in\calB}\normm{x_b}$ with $x_b=(x_i)_{i\in b}\in\bbR^{|b|}$. This function is partly smooth at $x$ with respect to the linear subspace
	\[
	\calM=T_x\eqdef\bBa{u\in\bbR^n:\supp_{\calB}(u)\subset\calS_{\calB},\quad \calS_{\calB}\eqdef\bigcup\{b:x_{b}\neq0\}}.
	\]
	In our experiment, we consider the true vector is of size $n=128$ with $2$ nonzero blocks of size $8$ each. The number of measurements is $m=0.5\times(2 \times 8)^2\times\log(128)$ quadratic measurements. We also take $\lambda=10^{-5}$. The results are shown in Figure~\ref{fig: solving_l12}, and they are consistent with the discussion for the \onenorm.
	\begin{figure}[htbp]
		\centering
		\includegraphics[trim={0cm 8cm 11cm 8.75cm},clip,width=0.35\linewidth]{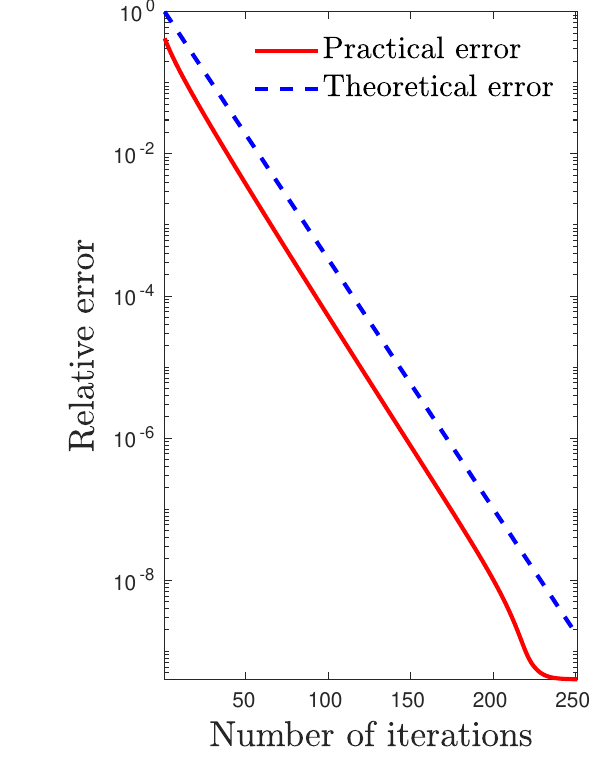}
		\caption{Phase retrieval by solving \eqref{eq:phase_retrieval} with the $\ell_{1,2}-$norm regularizer.}
		\label{fig: solving_l12}
	\end{figure}
\end{example}

\begin{example}\textbf{(Analysis-type prior).} Let $R_0\in\PSF{D x}{\calM_0}$ where $D:\bbR^n\rightarrow\bbR^p$ is a linear operator. When $D$ satisfies an appropriate transversality condition, then $R\eqdef R_{0}\circ D$ is partly smooth with respect to $\calM=\{u\in\bbR^n: D u\in\calM_0\}$. 
	
	The anisotropic total variation is a particular case where $R_0$ is the \onenorm and $D$ is a finite-difference operator with appropriate boundary conditions. It is polyhedral and partly smooth at $x$ relative to the linear subspace
	\[
	\calM=T_{x}\eqdef\{u\in\bbR^n:\supp(D u)\subset\supp(D x)\}.
	\] 
		In our experiment here, the original vector $\avx$ is piecewise constant with $s=12$ randomly placed jumps. The number of measurements is $m=0.5\times s^2\times\log(n)$. The regularizer is the total variation. Since the proximity operator of the latter is not explicit, we used the maxflow algorithm of \cite{ChambolleTV11} to compute it. The results are depicted in Figure~\ref{fig: solving_tv}. The left plot shows the original (dashed line) and the recovered vector (solid line). The right plot shows the evolution of the relative error vs iterations where again, a linear convergence behaviour is observed with a predicted rate that is very close to the observed one. 
		\begin{figure}[htbp]
			\centering
			\includegraphics[trim={0cm 7cm 0cm 7cm},clip,width=.75\linewidth]{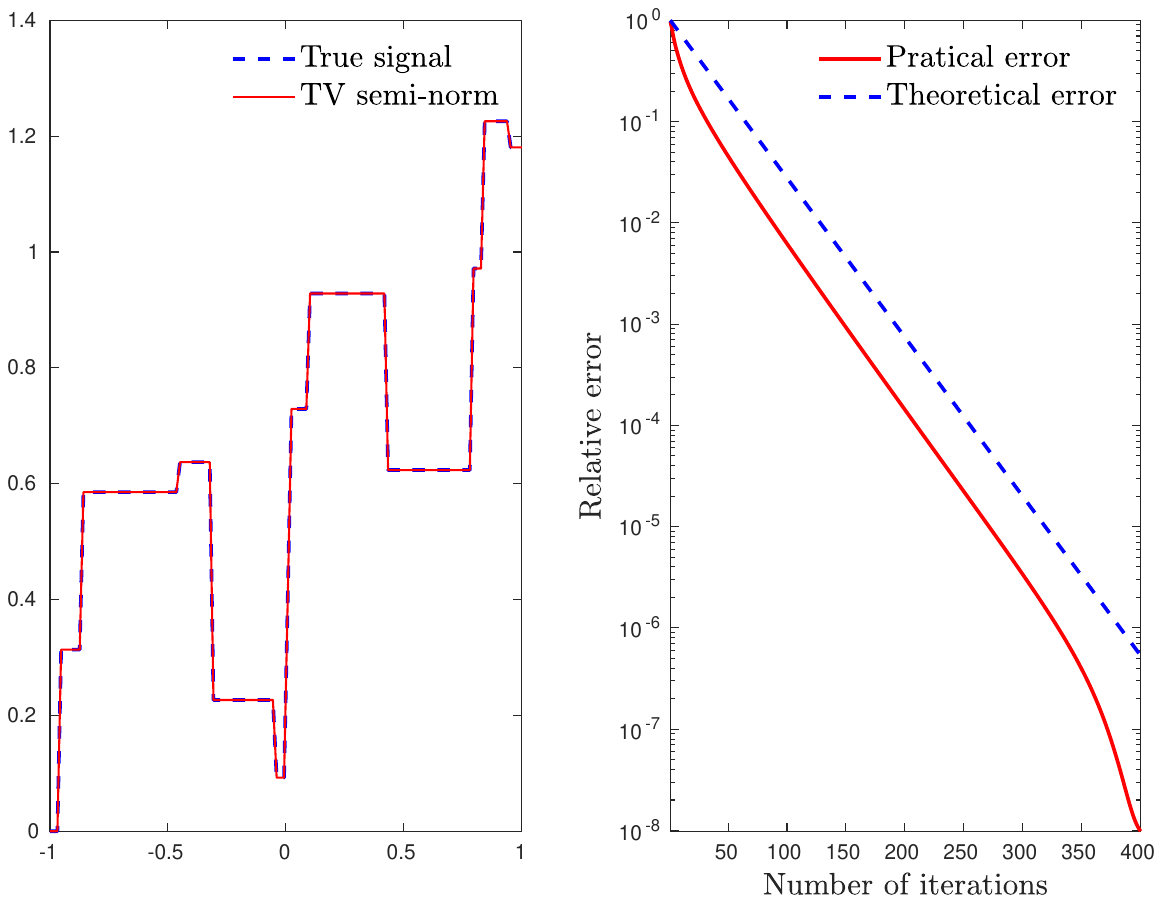}
			\caption{Phase retrieval by solving \eqref{eq:phase_retrieval} with the TV semi-norm.}
			\label{fig: solving_tv}
		\end{figure}
	\end{example}

	\begin{example}\textbf{(Wavelet synthesis-type prior).} 
		We here cast the phase retrieval problem as
		\begin{equation}\label{eq:tv_synthesis}
			\min_{v \in \bbR^p} \Phi(v) \eqdef \frac{1}{4m}\normm{y-|AW v|^2}^2+\lambda \norm{v}{1}, \quad \lambda>0 ,
		\end{equation}
		where $W$ is a wavelet synthesis operator. The reconstructed vector is given by $x=W v$. When $W$ is orthonormal, this is equivalent to the analysis-type formulation with $D=W^\top$. This is not anymore the case when $W$ is redundant. 
		
		In this experiment, we will use the shift-invariant wavelet dictionary with the Haar wavelet, which is closely related to the total variation regularizer for 1D signals; see \cite{steidl2004equivalence,waldspurger_phase_2015}. We take the same number of jumps and measurements as in the previous example. The results are shown in Figure~\ref{fig: solving_tv_syn}.
		
		\begin{figure}[htbp]
			\centering
			\includegraphics[trim={0cm 7cm 0cm 7cm},clip,width=.75\linewidth]{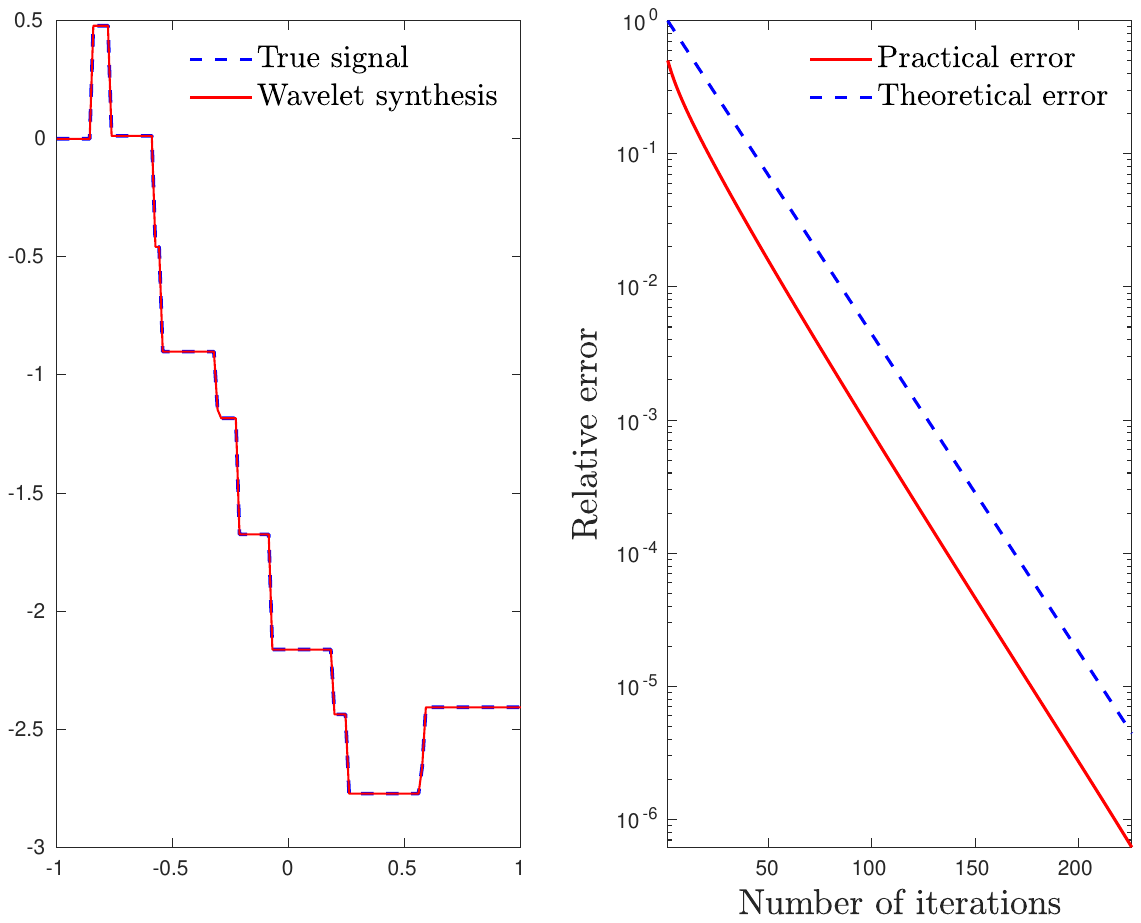}
			\caption{Phase retrieval with the synthesis prior formulation.}
			\label{fig: solving_tv_syn}
		\end{figure}
	\end{example}

\begin{appendices}\label{sec:appendix}

\section{Bregman Toolbox}\label{sec:bregman-tools}
\paragraph{Definition and properties}
Let us start with the definition of a Legendre function. 
\begin{definition}\label{def:legendre}\cite[Chapter 26]{rockafellar_convex_1970}\textbf{(Legendre function)} 
	Let $\phi\in\Gamma_0(\bbR^n)$ such that $\inte(\dom(\phi))\neq\emptyset$.  $\phi$ is called
	\begin{enumerate} [label=(\roman*)]
		\item \textit{essentially smooth} if it is differentiable on $\inte(\dom(\phi))$ with $\normm{\nabla\phi(\xk)}\to\infty$ for every sequence $\seq{\xk}$ of $\inte(\dom(\phi))$ converging  to a boundary point of $\dom(\phi)$.
		\item  \textit{essentially strictly convex}  if it is strictly convex on every convex subset of \[\dom~\partial\phi\eqdef\enscond{x}{\partial\phi(x)\neq\emptyset}.\]
	\end{enumerate} 
	A Legendre function is essentially smooth and strictly convex. 
\end{definition}
\begin{remark}\label{rmk: Legend_func}\cite[Theorem 26.5]{rockafellar_convex_1970}{\ }
	\begin{itemize}
		\item Let us notice that a function is Legendre if and only if its conjugate $\phi^*$ is of Legendre. 
		\item We also have that $\dom\partial\phi=\inte(\dom(\phi))$, $\partial\phi=\emptyset, \forall x\in\bd(\dom(\phi)) $ and $\forall x\in\inte(\dom(\phi))$ we have $\partial\phi(x)=\{\nabla\phi(x)\}$ and $\nabla\phi$ is a bijection from $\inte(\dom(\phi))$ to $\inte(\dom(\phi))^*$ with $\nabla\phi^*=(\nabla\phi)^{-1}$.  
	\end{itemize}
\end{remark}

For any function  $\phi:\bbR^n \to \overline{\bbR}$, we define a proximity measure associated with $\phi$.
\begin{definition}\textbf{(Bregman divergence)} The General  Bregman divergence associated with $\phi$  is
	\begin{equation}
		D_{\phi}^v(x,y)\eqdef \left\{\begin{aligned} &\phi(x)-\phi(y)-\pscal{v;x-y}, &\si (x,y)\in \left(\dom(\phi)\times \inte(\dom(\phi))\right),v\in\partial\phi(y)  ,\\
			&+\infty & \odwz. \end{aligned}\right.
	\end{equation}
\end{definition}
\begin{remark}{\ }
	When  $\phi$  is Legendre or simply sufficiently smooth on $\inte(\dom(\phi))$, we recover the classical definition \ie,
	\begin{equation}
		D_{\phi}(x,y)=\phi(x)-\phi(y)-\pscal{\nabla  \phi(y),x-y}.
	\end{equation}
	If $\phi(x)=\frac{1}{2}\normm{x}^2$, the Bregman divergence is the usual euclidean distance $D_{\phi}(x,y)=\frac{1}{2}\normm{x-y}^2.$ 
	This proximity measure is not a distance (it's not symmetric in general for instance). 
\end{remark}
Throughout the rest of the work, we use the following properties of the Bregman divergence.
\begin{proposition}\label{pro:bregman} \textbf{(Properties of the Bregman distance)}
	\begin{enumerate} [label=(\roman*)]
		\item \label{pro:bregman1} $D_\phi$ is nonnegative if and only if $\phi$ is convex. If in addition, $\phi$ is strictly convex, $D_\phi$ vanishes if and only if its arguments are equal. \label{pp:bregman1}
		
		\item \label{pro:bregman2}
		Linear additivity: for any $\alpha,\beta \in \bbR$ and any functions $\phi_1$ and $\phi_2$ sufficiently smooth,  we have
		\begin{equation}\label{linear}
			D_{\alpha \phi_1+\beta \phi_2}(x,u)=\alpha D_{\phi_1}(x,u) + \beta D_{\phi_2}(x,u),
		\end{equation}
		for all $(x,u)\in \left(\dom \phi_1\cap\dom \phi_2\right)^2$ such that both $\phi_1$ and $\phi_2$ are differentiable at $u$.
		\item \label{pro:bregman3} The three-point identity:  For any $x \in \dom(\phi)$ and $u,z\in \inte(\dom(\phi))$, we have 
		\begin{equation}\label{3poin}
			D_{\phi}(x,z)-D_{\phi}(x,u)-D_{\phi}(u,z)=\pscal{\nabla\phi(u)-\nabla\phi(z);x-u}.
		\end{equation}
		
		\item \label{pro:bregman4}
		Suppose that $\phi$ is also $C^2(\inte(\dom(\phi)))$ and $\nabla^2 \phi(x)$ is positive definite for any $x \in \inte(\dom(\phi))$. Then for every convex compact subset $\Omega\subset\inte(\dom(\phi))$, there exists $0<\theta_{\Omega}\leq\Theta_{\Omega}<+\infty$ such that for all $x,u \in \Omega$,
		\begin{equation}\label{eq:striconv}
			\frac{\theta_{\Omega}}{2}\normm{x-u}^2\leq D_{\phi}(x,u)\leq\frac{\Theta_{\Omega}}{2}\normm{x-u}^2.
		\end{equation}
	\end{enumerate} 
\end{proposition}

\paragraph{Regularity of functions}\label{sec:regul-func}
The following definition extends the classical gradient Lipschitz continuity property to the Bregman setting, this notion is named "relative smoothness" and  is important to the analysis of optimization problems that are differentiable but lack of gradient Lipschitz-smoothness. The earliest reference to this notion can be found in an economics paper \cite{birnbaum2011distributed} where it is used to address a problem in game theory involving fisher markets. Later on it was developed in \cite{bauschke_descent_2016,bolte_first_2017} and then in \cite{lu2018relatively}, although first coined relative smoothness in \cite{lu2018relatively}. Let $\phi \in \Gamma_0(\bbR^n) \cap C^1(\inte(\dom(\phi)))$, and $g$ be a proper and lower semicontinuous function such that $\dom(\phi) \subset \dom(g)$.
\begin{definition}\label{smoothadaptable}\textbf{($L-$relative smoothness)} Let $g \in C^1(\inte(\dom(\phi)))$,
	$g$ is called $L-$smooth relative to $\phi$ on $\inte(\dom(\phi))$ if there exists $L>0$ such that $L\phi-g$ is convex on $\inte(\dom(\phi))$, \ie 
	\begin{equation}\label{eq:smoothadaptable}
		D_g(x,u) \leq LD_{\phi}(x,u) \qforallq (x,u) \in \dom(\phi) \times \inte(\dom(\phi)) .
	\end{equation}
\end{definition}

When $\phi$ is the energy entropy, \ie $\phi=\frac{1}{2}\normm{\cdot}^2$, one recovers the standard descent lemma implied by Lipschitz continuity of the gradient of $g$.

In a similar way, we also extend the standard local strong convexity property to a relative version \wrt to an entropy or kernel $\phi$.

\begin{definition}\label{relativeconvex}\textbf{(Local relative strong convexity)}
	Let $\C$ be a non-empty subset of $\dom(\phi)$. Let $g \in C^1(\inte(\dom(\phi)))$, for $\sigma > 0$ we say that $g$ is $\sigma$-strongly convex on $\C$ relative to $\phi$ if 
	\begin{equation}
		D_g(x,u)\geq \sigma D_{\phi}(x,u) \qforallq x \in \C \tandt u \in \C \cap \inte(\dom(\phi)) .
	\end{equation}
\end{definition}
When $\C=\dom(\phi)$, we get the idea of global relative strong convexity. If $\phi$ is the energy entropy (\ie $\phi=\frac{1}{2}\normm{\cdot}^2$), one recovers the standard definition of (local/global) strong convexity.

The idea of global (\ie $\C=\dom(\phi)$) relative strong convexity has already been used in the literature, see \eg  \cite[Proposition~4.1]{Teboulle18} and \cite[Definition~3.3]{Bauschke19}. Its local version was first proposed in \cite{Silveti22}. When $\phi$ is the energy entropy (\ie $\phi=\frac{1}{2}\normm{\cdot}^2$), one recovers the standard definition of (local/global) strong convexity. Relation of global relative strong convexity to gradient  dominated inequalities, which is an essential ingredient to prove global linear convergence of mirror descent, was studied in \cite[Lemma~3.3]{Bauschke19}.

Let us give the following useful lemma which compare the Bregman divergences of smooth functions.
\begin{lemma}\label{lem:bregcomp}
	Let $g,\phi\in C^2(\bbR^n)$. If $\forall u\in\bbR^n$, $\nabla^2g(u) \lon \nabla^2\phi(u)$ for all $u$ in the segment $[x,z]$, then,
	\begin{equation}
		D_g(x,z)\leq D_{\phi}(x,z) .
	\end{equation}
\end{lemma}
\begin{proof}
	The result comes from the Taylor-MacLaurin expansion. Indeed we have $\forall x,z\in \bbR^n$ 
	\begin{align*}
		D_g(x,z)
		&=g(x)-g(z)-\pscal{\nabla g(z),x-z} \\
		&=\int_0^1(1-\tau)\pscal{x-z,\nabla^2g(z+\tau(x-z))(x-z)}d\tau,
	\end{align*}
	and thus
	\begin{multline*}
		D_{\phi}(x,z)-D_{g}(x,z) = \\ \int_0^1(1-\tau)\pscal{x-z,\Ppa{\nabla^2\phi(z+\tau(x-z))-\nabla^2g(z+\tau(x-z))}(x-z)}d\tau .
	\end{multline*}
	The positive semidefiniteness assumption implies the claim.
\end{proof}

\paragraph{Triangle scaling property}
Here, we introduce the triangle scaling property (TSP) \cite{hanzely_accelerated_2021} for Bregman distances. 
\begin{definition}\label{df:triangle-scaling} Let $\phi$ be a Legendre function. The Bregman distance generated by $\phi$ has the triangle scaling property if there is a constant $\kappa>0$ such that for all $x, y,z\in \ri{\dom(\phi)}$,
	\begin{equation}\label{eq:triangle-scaling}
		D_{\phi}((1 - a)x + a y, (1 - a)x +  az) \leq a^{\kappa} D_{\phi}(y,z),\quad \forall a\in [0,1]. 
	\end{equation}
	We call $\kappa$ the uniform triangle scaling exponent (TSE) of $D_{\phi}$.
\end{definition}
There is a large class of functions that satisfy this property, here are some specific examples. 
\begin{itemize} 
	\item \textit{Euclidean distance.} When $\phi$ is the energy and thus $D_{\phi}(x,y)=\frac{1}{2}\normm{x-y}^2.$ The squared Euclidean distance has a uniform TSE $\kappa=2$.   
	\item \textit{Bregman divergence induced by strongly convex and smooth functions.} If $\phi$ is $\sigma_{\phi}-$strongly convex and $L-$smooth over its domain then \eqref{eq:triangle-scaling} hold with $\kappa=2$ if the right-hand side is multiplied by the condition number $L/\sigma_{\phi}$.
	\item \textit{Bregman geometry based on polynomial kernel.} Polynomial functions of the form \(\phi(x) = \frac{1}{p}\|x\|^p\) for some \(p \geq 2\), the global TSE for the induced Bregman divergence can be less than 1 for \(p > 2\). However, the modified reference function \(\phi(x) = \frac{1}{2}\|x\|^2 + \frac{1}{p}\|x\|^p\) for \(p \geq 4\) has  a coefficient \(\kappa > 1\), or \(\kappa = 2\) with an additional factor on the right-hand side of \eqref{eq:triangle-scaling}, over a bounded domain. It turns that this choice of $\phi$ is precisely the one that we will make for phase retrieval.
\end{itemize}
%

\section{KL Functions}\label{sec:ch2-kltools}
This section encompasses all the essential components required for the axiomatization of convergence for KL functions and therefore can be skipped by an experienced reader.
We start by defining the non-smooth KL property which is an additional assumption on the class of functions that we consider. This property gives a hint about the geometric bearing of the function near the point where it is satisfied. 
\begin{definition}\textbf{(Non-smooth KL property)}\label{def: KL_property} A proper and  lower semicontinuous function $g:\bbR^n\to\overline{\bbR}$ has the KL property at a point $\xpa\in\dom(g)$ if there exists a neighborhood $U_{\xpa}$, $\eta>0$ and a concave real-valued function $\varphi\in C^1([0,\eta[)$, with $\varphi(0)=0$ and $\varphi'>0$, such that
	\[  
	\varphi'\Ppa{g(x)-g(\xpa)}\dist(0,\partial g(x))\geq 1,\quad \forall x \in U_{\xpa} \cap \enscond{x\in\bbR^n}{g(\xpa) < g(x) < g(\xpa)+\eta}.
	\]
	If $g$ has the KL property at each point of $\dom(g)$, $g$ is called a KL function.
\end{definition}
$\varphi$ is known as the desingularizing function. The KL property is also closely related to error bounds and the broader notion of ``(sub)metric regularity''. We refer the reader to \cite{ioffe2017variational} and \cite{bolte_characterizations_2009} for a detailed study of these notions. In general, it is not obvious to check whether a given function is KL or not. Actually, this is a very deep question that has been studied at the interface of analysis and algebraic geometry. For smooth functions, it has been shown that semi-algebraic and sub-analytic are {\L}ojaciewicz in the seminal works of \cite{loj1,loj2,kurdyka1998gradients}. This has been extended to the nonsmooth case and then widely studied in the scope of optimization in \cite{bolte2006clarke,bolte_lojasiewicz_2007,bolte_characterizations_2009}. For instance, functions definable on o-minimal structures are KL. This cover most functions studied in practice, and for instance those in this manuscript.

The following uniformization of the KL property will be very useful; see \cite[Lemma~6]{bolte_proximal_2014}.
\begin{lemma}[Uniformized KL property]\label{lem:uniform-KL}
	Let $\Omega$ be a compact set, and let $g: \mathbb{R}^n \to \overline{\bbR}$ be a proper and closed function. Assume that $g$ is constant on $\Omega$ and satisfies the KL property at each point of $\Omega$. Then, there exist $\eps>0$, $\eta>0$, if there exists $\eta>0$ and a concave  real-valued function $\varphi\in C^1([0,\eta[)$ with $\varphi(0)=0,\varphi'>0$ such that for all $\xpa$ in $\Omega$, one has
	\[
	\varphi'(g(x) - g(\xpa)) \dist(0, \partial g(x)) \geq 1,
	\]
	and all $x \in \bbR^n$ such that $\dist(x, \Omega) < \eps$ and $g(\xpa) < g(x) < g(\xpa)+\eta$.
\end{lemma}

For the convergence of our inertial Bregman proximal gradient (IBPG) algorithm, we will use a general convergence mechanism as first axiomatized in \cite{attouch_convergence_2013} for descent algorithms and generalized in \cite{bolte_proximal_2014} on the so-called PALM algorithm, so that it can be used and applied to any given algorithm such as ours (see also, \eg \cite[Appendix 6]{bolte_first_2017} for a self-contained presentation). The main goal is to prove that the whole sequence $\seq{\xk}$ generated by IBPG, converges to a critical point. For that purpose, considering a  Lyapunov function $\Psi$ associated to IBPG, it has to satisfy the following three key conditions.  
\begin{definition}[Descent-like method]\label{def:condition} 
	A sequence $\seq{\xk}$ is called descent-like for the function $\Psi$ if the following conditions hold:
	\begin{enumerate} [label=(C.\arabic*)]
		\item\label{def:condition1} \textit{Sufficient decrease condition.} There exists a positive scalar $\rho_1$ such that
		\[
		\rho_1\normm{\xk-\xkm}^2\leq \Psi(\xk,\xkm) -\Psi(\xkp,\xk),\quad \forall k\in\bbN. 
		\]
		\item\label{def:condition2}  \textit{Relative error condition.} There exists $K\in\bbN$ and $\rho_2>0$ such that $\forall k\geq K$, there exists $v_{k+1}\in\partial\Psi(\xkp,\xk)$ such that
		\[
		\normm{v_{k+1}}\leq\rho_2\Ppa{\normm{\xkp-\xk}+\normm{\xk-\xkm}} .
		\] 
		\item \label{def:condition3} \textit{Continuity condition.} Let $x^{\ast}$ be a limit point of a subsequence $(\xk)_{k\in\calK\subset\bbN}$ then we have that  $\limsup\limits_{k\in\calK\subset\bbN}\Psi(\xk,\xkm)\leq\Psi (x^{\ast}) \eqdef \Psi (x^{\ast},x^{\ast})$. 
	\end{enumerate}
\end{definition}
Condition \ref{def:condition1} is intended to model a descent property of the Lyapunov function, and hence a dissipation of the energy $\Psi$. \ref{def:condition2}\footnote{The original version of this condition in \cite{attouch_convergence_2013} involves only the first term in the bound. The reasoning however remains the same with this version of the inequality; see \eg \cite{Bot14,liang_multi-step_2016,mukkamala_convex-concave_2020}.} originates from the well-known fact that most algorithms in optimization generate sequences via exact or inexact minimization of subproblems and condition \ref{def:condition2} reflects relative inexact optimality conditions for such minimization subproblems. Condition~\ref{def:condition3} is a weak requirement which, in particular, holds when $\Psi$ is continuous. However, the latter is not mandatory in general as the nature of the algorithm (IBPG here) will force the sequences to comply with \ref{def:condition3} under a simple lower semicontinuity assumption.

Equipped with Definition~\ref{def:condition}, and when $\Psi$ satisfies the KL property, the following global convergence result holds true.
\begin{theorem}[Global convergence]\label{thm:KLglobalconv} 
	Let $\seq{\xk}$ be a bounded sequence generated by a descent-like method for $\Psi$. If $\Psi$ satisfies the KL property, then the sequence $\seq{\xk}$ has finite length, \ie, $\sum_{k \in \bbN}\normm{\xkp-\xk} < +\infty$ and it converges to $x_\star \in \crit{\Psi}$.
\end{theorem}	

\section{Proof of Global Convergence}\label{pr:thm:Global_conv_ABPG} 
\begin{lemma}\label{lem:link-seq} Let $\seq{\xk}, \seq{\yk}$ and $\seq{\zk}$ be generated by IBPG (Algorithm~\ref{alg:BPGBT}). Then the following holds true
	\begin{enumerate}[label=(\roman*)]
		\item  If $\seq{\xk}$ is bounded, then the other sequences are also bounded.
		\item  If $\seq{\xk}$ has a limit, then the remaining sequences also converge to the same limit.
	\end{enumerate}
\end{lemma}
\begin{proof} 
	\begin{enumerate}[label=(\roman*)]
		\item $\zk$ being a convex combination of $\xk$ and $\xkm$, the conclusion is immediate. The same holds also for $\yk$.
		\item Suppose that $x_k \to x^{\ast}$. We have by definition
		\[
		\normm{\zk - x^{\ast}} \leq \normm{\xk-x^{\ast}} + \normm{\xk - \xkm} .
		\]
		Passing to the limit we get that $\zk \to x^{\ast}$. Moreover, 
		\[
		\normm{\yk - x^{\ast}} \leq (1-a_k) \normm{\zk - x^{\ast}} + a_k \normm{\xk - x^{\ast}} \leq \normm{\zk - x^{\ast}} + \normm{\xk - x^{\ast}} .
		\]
		and thus $\yk \to x^{\ast}$.
	\end{enumerate}
\end{proof}

Our global convergence analysis will be based on a Lyapunov analysis with the energy function $\Psi_k$ on $\bbR^{3n}$ defined as 
\[
\Psi_k(\xk,\xkm) = \Phi(\xk) - \inf \Phi + a_{k-1}^{1-\kappa} L D_{\psi}(\xk,\xkm) ,
\]
where $\kappa\in ]1,2]$ is the TSE parameter of $\psi$. The subscript $k$ underscores the fact that $\Psi_k$ depends on $a_{k-1}$. Observe that $\Psi_k$ is non-negative. The first part of $\Psi_k$ corresponds to the potential energy of IBPG seen as a dissipative (discrete) dynamical system. The second term, which captures how the iterates remain close to each other, can be interpreted as a discrete Bregman version of the kinetic energy (involving the discrete velocity) of the system. The following lemma shows that $\Psi_k$ is indeed a Lyapunov function for IBPG. 
\begin{lemma}\label{lem:lyapnuovineq}
	Under Assumptions~\ref{assump_A} and \ref{assump_B4}, there exists $\nu \in ]0,1]$ such that the sequences generated by Algorithm~\ref{alg:BPGBT} satisfy $\forall k\geq0$
	\begin{equation}\label{eq:descent-lemma}
		\Psi_{k+1}(\xkp,\xk) \leq \Psi_k(\xk,\xkm) - \nu L D_{\psi}(\xk,\xkm) - \Ppa{\oa^{1-\kappa}-1}LD_{\psi}(\xkp,\yk) .
	\end{equation}
\end{lemma}
\begin{proof} 
	By $L$-smoothness of $F$ relative to $\psi$, we have 
	\begin{align*}
		\Phi(\xkp) 
		&\leq F(\yk)+\pscal{\nabla F(\yk),\xkp-\yk}+G(\xkp) + L D_{\psi}(\xkp,\yk) . \label{eq:cont-cond}\numberthis
	\end{align*}
	Observe that 
	\begin{equation}\label{eq:defxk+1}
		\xkp \in \Argmin_{x \in \bbR^n} F(y_k) + \pscal{\nabla F(\yk),x-\yk}+G(x)+\frac{1}{\gamma_k}D_{\psi}(x,\yk) .
	\end{equation}
	Then convexity of $G$ gives that $\forall x \in \bbR^n$ (see \eg \cite[Lemma~3.2]{chen_convergence_1993})
	\begin{multline}\label{eq:ineqdefxk+1}
		F(y_k) + \pscal{\nabla F(\yk),\xkp-\yk}+G(\xkp) + \frac{1}{\gamma_k}D_{\psi}(\xkp,\yk) \leq F(y_k) + \pscal{\nabla F(\yk),x-\yk}+G(x)\\
		+\frac{1}{\gamma_k}D_{\psi}(x,\yk) - \frac{1}{\gamma_k}D_{\psi}(\xkp,x)  .
	\end{multline}
	Inserting this into \eqref{eq:cont-cond}, we get
	\begin{align*}
		\Phi(\xkp) 
		&\leq F(\yk) + \pscal{\nabla F(\yk),x-\yk}+G(x) + \frac{1}{\gamma_k}D_{\psi}(x,\yk) - \frac{1}{\gamma_k}D_{\psi}(\xkp,x) - \Ppa{\frac{1}{\gamma_k}-L}D_{\psi}(\xkp,\yk) \\
		&= \Phi(x) - D_F(x,\yk) + \frac{1}{\gamma_k}D_{\psi}(x,\yk) - \frac{1}{\gamma_k}D_{\psi}(\xkp,x) - \Ppa{\frac{1}{\gamma_k}-L}D_{\psi}(\xkp,\yk) \\
		&\leq \Phi(x) + \Ppa{\frac{1}{\gamma_k}+L}D_{\psi}(x,\yk) - \frac{1}{\gamma_k}D_{\psi}(\xkp,x) - \Ppa{\frac{1}{\gamma_k}-L}D_{\psi}(\xkp,\yk) ,
	\end{align*}
	where we used again $L$-smoothness of $F$ relative to $\psi$.
	Applying this inequality at $x=\xk$, we obtain
	\begin{align*}
		\Phi(\xkp)+\frac{1}{\gamma_k}D_{\psi}(\xkp,\xk)
		\leq \Phi(\xk) + \Ppa{\frac{1}{\gamma_k}+L}D_{\psi}(\xk,\yk) - \Ppa{\frac{1}{\gamma_k}-L}D_{\psi}(\xkp,\yk) .   
	\end{align*}
	We now use the TSP property twice to get
	\begin{align*}
		D_{\psi}(\xk,\yk) &= D_{\psi}((1-a_k)\xk+a_k\xk,(1-a_k)\zk+a_k\xk) \leq (1-a_k)^{\kappa}D_{\psi}(\xk,\zk) \qandq \\
		D_{\psi}(\xk,\zk) &= D_{\psi}((1-a_{k-1})\xk+a_{k-1}\xk,(1-a_{k-1})\xkm+a_{k-1}\xk) \leq (1-a_{k-1})^{\kappa}D_{\psi}(\xk,\xkm) .
	\end{align*}
	Combining the above inequalities, we arrive at 
	\begin{multline*}
		\Phi(\xkp)+\frac{1}{\gamma_k}D_{\psi}(\xkp,\xk)
		\leq \Phi(\xk) + \Ppa{\frac{1}{\gamma_k}+L}(1-a_k)^{\kappa}(1-a_{k-1})^{\kappa}D_{\psi}(\xk,\xkm) \\
		- \Ppa{\frac{1}{\gamma_k}-L}D_{\psi}(\xkp,\yk) .
	\end{multline*}
	Thus, in view of the choice of the parameters, there exists $\nu ]0,1]$ such that 
	\begin{align*}
		\Phi(\xkp)+a_k^{1-\kappa}LD_{\psi}(\xkp,\xk)
		&\leq \Phi(\xk) + a_{k-1}^{1-\kappa}L D_{\psi}(\xk,\xkm) - \nu a_{k-1}^{1-\kappa}L D_{\psi}(\xk,\xkm) \\
		&- \Ppa{a_k^{1-\kappa}-1}LD_{\psi}(\xkp,\yk) \\
		&\leq \Phi(\xk) + a_{k-1}^{1-\kappa}L D_{\psi}(\xk,\xkm) - \nu L D_{\psi}(\xk,\xkm) \\
		&- \Ppa{\oa^{1-\kappa}-1}LD_{\psi}(\xkp,\yk) .
	\end{align*}
	Subtracting $\inf \Phi$ on both sides, we get the claimed inequality. 
\end{proof}

Capitalizing on the above descent property of the Lyapunov function $\Psi_k$, we get some preliminary convergence properties\footnote{Only convexity of $\psi$ is needed for claim \ref{pro:seq_conv-i} to hold.}. 
\begin{proposition}\label{pro:seq_conv} 
	Under Assumptions~\ref{assump_A} and \ref{assump_B1}-\ref{assump_B4}, the sequences generated by Algorithm~\ref{alg:BPGBT} (IBPG) are such that:  
	\begin{enumerate}[label=(\roman*)]
		\item \label{pro:seq_conv-i} The sequence $\Ba{\Psi_k(\xk,\xkm)}_{k\in\bbN}$ is decreasing and thus $\lim_{k \to +\infty} \Psi_k(\xk,\xkm) \geq 0$ exists.
		\item \label{pro:seq_conv-ii} $\sum_{k \in \bbN}\normm{\xk-\xkm}^2<\infty$ and thus $\lim_{k \to \infty} \normm{\xk-\xkm}=0$. 
		\item \label{pro:seq_conv-iii} We have the rate
		\[
		\min\limits_{0 \leq i\leq k} \normm{x_i-x_{i-1}}^2\leq \frac{2(\sigma_{\psi}\nu L)^{-1}\Psi_0(x_0,x_{-1})}{k+1} .
		\]
	\end{enumerate}
\end{proposition}
\begin{proof}
	\begin{enumerate}[label=(\roman*)]
		\item The first claim comes from Lemma~\ref{lem:lyapnuovineq} and non-negativity of $D_\psi$ thanks to convexity of $\psi$. The existence of the limit then follows as $\Psi_k$ is non-negative and decreasing (recall that $\Phi$ is bounded from below, $a_k \geq \ua > 0$ and $\kappa \in ]0,1]$). 
		
		\item We use $\sigma_{\psi}$-strong convexity of $\psi$ and then sum \eqref{eq:descent-lemma} dropping the last non-negative term to get
		\begin{align*}
			\frac{\sigma_{\psi}\nu L}{2}\sum_{i=0}^k \normm{x_i - x_{i-1}}^2
			&\leq \nu L\sum_{i=0}^k D_{\psi}(x_i,x_{i-1}) \\
			&\leq \Psi_0(x_0,x_{-1}) - \Psi_k(\xk,\xkm) \leq  \Psi_0(x_0,x_{-1}) . \numberthis \label{eq:sumzx}
		\end{align*}
		Passing to the limit as $k \to +\infty$ we get the summability claim.
		
		\item We have
		\[
		(k+1)\min\limits_{0 \leq i\leq k} \normm{x_i-x_{i-1}}^2 \leq \sum_{i=0}^k \normm{x_i-x_{i-1}}^2 . 
		\] 
		Combining this with \eqref{eq:sumzx}, we conclude. 
	\end{enumerate}
\end{proof}

To prove global convergence, it is sufficient to show that IBPG is a descent-like method according to Definition~\ref{def:condition} and then to invoke Theorem~\ref{thm:KLglobalconv}. For this we need to a construct an appropriate function $\Psi$ that verifies the conditions of Definition~\ref{def:condition}. The sequence of functions $\Psi_k$ could do the job if $a_k$ is taken fixed, say equal to $a \in [\ua,\oa]$, for all $k \geq K$, where $K$ is arbitrarily large (see the discussion in Remark~\ref{rmk: global_conv}). We therefore consider the energy function
\[
\Psi(\xk,\xkm)=
\begin{cases} 
	\Psi_k(\xk,\xkm)&\si\quad k < K ,\\
	\Phi(\xk) - \inf \Phi + a^{1-\kappa}LD_{\psi}(\xk,\xkm) & \odwz .
\end{cases}
\]
Observe first that $\Psi$ is KL since both $\Phi$ and $\psi$ are. The following proposition shows that the sequence generated by IBPG is a descent-like sequence for the new Lyapunov function $\Psi$. 
\begin{proposition}\label{pro:grad_like_des} 
	Assume that Assumptions~\ref{assump_A} and Assumptions~\ref{assump_B} hold. Let $\seq{\xk}$ be a bounded sequence generated by Algorithm~\ref{alg:BPGBT}. Then $\seq{\xk}$ is a gradient-like descent sequence. Moreover the set of cluster points of $\seq{\xk}$ is a nonempty compact set of $\crit{\Phi}$.    
\end{proposition}
\begin{proof}{\ } 
	\begin{itemize}
		\item \textit{ Sufficient decrease condition.} From Lemma~\ref{lem:lyapnuovineq} $\sigma_{\psi}$-strong convexity of $\psi$, we have $\forall k\in\bbN$
		\begin{align*}
			\Psi(\xkp,\xk) 
			&\leq \Psi(\xk,\xkm) - \nu L D_{\psi}(\xk,\xkm) \\
			&\leq \Psi(\xk,\xkm) - \frac{\sigma_{\psi}\nu L}{2} \normm{\xk-\xkm}^2
		\end{align*}
		which shows \ref{def:condition1} in Definition~\ref{def:condition}.
		\item\label{pr_relative-error} \textit{Relative error condition.} We have to show \ref{def:condition2} in Definition~\ref{def:condition}. 
		
		$\psi$ is $C^2$ hence $D_\psi(\cdot,\cdot)$ is $C^1$ jointly in its arguments. The sum rule of the limiting subdifferential applies in this case and tells us that for $k \geq K$, we have
		\begin{align}
			\partial\Psi\pa{\xkp,\xk}=
			\begin{pmatrix}
				\nabla F(\xkp)+\partial G(\xkp)+a^{1-\kappa}L(\nabla\psi(\xkp)-\nabla\psi(\xk)) \\
				- a^{1-\kappa}L\nabla^2\psi(\xk)(\xkp-\xk)
			\end{pmatrix}. \label{eq:subdifpsi}
		\end{align}
		From the update equation of $\xkp$ by IBPG, we have
		\begin{align}\label{eq:subdifGxk+1}
			\nabla\psi(\yk) - \gamma \nabla F(\yk) - \nabla\psi(\xkp) \in \gamma \partial G(\xkp) ,
		\end{align}
		where $\gamma = a^{\kappa-1}/L$. Set $v_{k+1} \eqdef (v_{k+1}^1,v_{k+1}^2)$ where
		\begin{align*}
			v_{k+1}^1 &= \Ppa{\nabla F(\xkp) - \nabla F(\yk)} + a^{1-\kappa}L\Ppa{\nabla\psi(\yk) - \nabla\psi(\xk)} \\
			\qandq 
			v_{k+1}^2 &= - a^{1-\kappa}L\nabla^2\psi(\xk)(\xkp-\xk) .  
		\end{align*}
		In view of \eqref{eq:subdifpsi} and \eqref{eq:subdifGxk+1}, we have
		\begin{align*}
			v_{k+1} \in \partial\Psi\pa{\xkp,\xk} .
		\end{align*}
		We shall now bound $\normm{v_{k+1}}^2=\normm{v_{k+1}^1}^2+\normm{v_{k+1}^2}^2$. Since $\seq{\xk}$ is bounded and $\nabla\psi$ is Lipschitz continuous on any bounded subset by \ref{assump_B3}, there exists $L_{\psi}$ such that
		\[
		\normm{v_{k+1}^2} \leq a^{1-\kappa} LL_\Psi\normm{\xkp-\xk} \leq \ua^{1-\kappa} LL_{\psi}\normm{\xkp-\xk} .
		\]  
		Moreover, $\seq{\xk}$ is also bounded by Lemma~\ref{lem:link-seq}. Assumption~\ref{assump_B3} then entails that there exist $L_F > 0$ such that    
		\begin{align*}
			\normm{v^1_{k+1}}
			&\leq L_F\normm{\xkp-\xk} + L_F\normm{\xk-\yk} + \ua^{1-\kappa}LL_{\psi}\normm{\xk-\yk} .
		\end{align*}
		where we used that $a \in ]0,1]$ and $\kappa \in ]1,2]$. Now, by definition of the iterates
		\[
		\xk - \yk = \xk - \zk - a(\xk-\zk) = (1-a)(\xk - \xkm - a(\xk-\xkm)) = (1-a)^2(\xk-\xkm) .
		\]
		Therefore
		\begin{align*}
			\normm{v^1_{k+1}}
			&\leq L_F\normm{\xkp-\xk} + (L_F + \ua^{1-\kappa}LL_{\psi})(1-a)^2\normm{\xk-\xkm} \\
			&\leq L_F\normm{\xkp-\xk} + (L_F + \ua^{1-\kappa}LL_{\psi})(1-\ua)^2\normm{\xk-\xkm} .
		\end{align*}
		Taking $\rho_2=L_F+(L_F + \ua^{1-\kappa}LL_{\psi})(1-\ua)^2+\ua^{1-\kappa} LL_{\psi}$, we get the claim.
		
		\item \textit{Continuity  condition.} Since $\seq{\xk}$ is bounded, its set of cluster points is a nonempty set. It is also compact set as the intersection of compact sets. Let us consider a subsequence $(\xkj)_{j\in\bbN}$ that converges to some limit $x^{\ast}$. From Proposition~\ref{pro:seq_conv}\ref{pro:seq_conv-i} and Lemma~\ref{lem:link-seq}, we have that $\seq{\xkjp}$ and $\seq{\ykj}$ converge to the same limit. Now arguing as in \eqref{eq:defxk+1}-\eqref{eq:ineqdefxk+1}, we get
		\begin{align*}
			G(\xkjp) 
			&\leq G(x^{\ast}) + \pscal{\nabla F(\ykj),x^{\ast}-\xkjp} - \frac{1}{\gamma}D_{\psi}(\xkjp,\ykj)
			+\frac{1}{\gamma}D_{\psi}(x^{\ast},\ykj) - \frac{1}{\gamma}D_{\psi}(\xkjp,x^{\ast})  \\
			&\leq G(x^{\ast}) + \pscal{\nabla F(\ykj),x^{\ast}-\xkjp} + \frac{1}{\gamma}D_{\psi}(x^{\ast},\ykj) .
		\end{align*}
		Passing to the limit as $j \to +\infty$ and using continuity of $D_{\psi}$ we get that
		\[
		\limsup_{j \to +\infty} G(\xkjp) \leq G(x^{\ast}) .
		\]
		Combining this with continuity of $F$ proves \ref{def:condition3} in Definition~\ref{def:condition}.
	\end{itemize} 
	Let $\xkj \to x^{\ast}$. We argued above that $\ykj \to x^{\ast}$ and $\xkjp \to x^{\ast}$. We then have from \eqref{eq:subdifGxk+1}, continuity of $\nabla \psi$ and $\nabla F$, and sequential closedness of $\nabla G$ that
	\[
	-\nabla F(x^{\ast}) \in \partial G(x^{\ast}) ,
	\]
	\ie $x^{\ast} \in \crit{\Phi}$.
\end{proof}

\paragraph{Proof of Theorem~\ref{thm:Global_conv_ABPG}}{\ }
\begin{proof}
	\begin{enumerate}[label=(\roman*)]
		\item This claim comes from Proposition~\ref{pro:seq_conv}.
		\item This is a consequence of Proposition~\ref{pro:grad_like_des} and Theorem~\ref{thm:KLglobalconv} after observing from \eqref{eq:subdifpsi} that $\crit{\Psi}=\{(x_\star,x_\star): \; x_\star\in\crit{\Phi}\}$.
		\item We now turn to proving convergence to a global minimizer. We introduce the following extended variable, $\wtilde{\xk}=(\xk,\xkm)\in\bbR^n\times\bbR^n$ such that $\Psi\pa{\xk,\xkm}=\Psi\pa{\wtilde{\xk}}$ for all $k\in\bbN$. Let us choose radius $r>\rho_2 >0$ such that \(\eta<\rho_1(r - \rho_2)^2\). Let us suppose that the initial point $\xo$ is chosen such that the following conditions hold,
		\begin{align}
			&\Phi(x^\star)=\Psi(\wtilde{x^\star})\leq\Psi(\wtilde{x_0}) < \Psi(\wtilde{x^\star})+\eta=\Phi(x^\star)+\eta \label{eq:local-conv-asump1}\\
			&\normm{\xo-\xsol} + 2\sqrt{ \frac{\Psi(\wtilde{x_0}) - \Psi(\wtilde{x^\star})}{\rho_1}} + \frac{\rho_2}{\rho_1}\varphi\BPa{\Psi(\wtilde{x_0})-\Psi(\wtilde{x^\star})}<\rho.\label{eq:local-conv-asump2}
		\end{align}
		The condition ~\ref{def:condition1}  combined with \eqref{eq:local-conv-asump1} imply that  for any $k\in\bbN, \Psi\pa{\wtilde{x^\star}}\leq\Psi\pa{\wtilde{\xkp}}\leq\Psi\pa{\wtilde{x_0}}<\Psi(\wtilde{x^\star})+\eta$, and moreover
		\begin{equation}\label{eq:bound-xk}
			\normm{\xkp-\xk}\leq\sqrt{\frac{\Psi\pa{\wtilde{\xkp}}-\Psi\pa{\wtilde{x_{k+2}}}}{\rho_1}}\leq \sqrt{\frac{\Psi\pa{\wtilde{\xkp}}-\Psi\pa{\wtilde{x^\star}}}{\rho_1}}.
		\end{equation}
		We deduce that if for any $k\in\bbN, \xk\in B(\xsol,\rho)$ then $\xkp\in B(\xsol,r)$. Indeed, by the triangle inequality 
		\begin{equation}\label{eq:xkp-xsol}
			\normm{\xkp-\xsol}\leq \normm{\xk-\xsol}+\sqrt{\frac{\Psi\pa{\wtilde{\xkp}}-\Psi\pa{\wtilde{x^\star}}}{\rho_1}}=\rho+(r-\rho)=r.
		\end{equation}
		It remains to show that $\forall k\in\bbN, \xk\in B(\xsol,\rho)$. We argue by induction. The triangle inequality gives 
		\begin{align*}
			\normm{x_1 - \xsol}\leq \normm{x_0-\xsol}+\sqrt{\frac{\Psi\pa{\wtilde{x_1}}-\Psi\pa{\wtilde{x^\star}}}{\rho_1}}\leq \normm{x_0-\xsol}+\sqrt{\frac{\Psi\pa{\wtilde{x_0}}-\Psi\pa{\wtilde{x^\star}}}{\rho_1}}<\rho,
		\end{align*}
		which means that $x_1 \in B(\xsol,\rho)$. We also have
		\begin{align*}
			\normm{\xkp-\xsol}&\leq \normm{z_0-\xsol}+ 2\normm{x_1-x_0}+\sum_{j=1}^k\normm{\xk-\xkm}, 
		\end{align*}
		Standard arguments with the KL inequality show that
		\begin{align}\label{eq:bound-sum-zk}
			\sum_{i=l+1}^{k}\normm{\xk-\xkm}&\leq \frac{\rho_2}{\rho_1}\varphi\BPa{\Psi(x_{l+1},x_{l})-\Psi(\wtilde{x^\star})} .
		\end{align}
		Applying this bound \eqref{eq:bound-sum-zk} with $l=0$ and combining  with \eqref{eq:local-conv-asump2} yields
		\begin{align*}
			\normm{\xkp-\xsol}&\leq \normm{x_0-\xsol}+ 2\sqrt{\frac{\Psi\pa{\wtilde{x_0}}-\Psi\pa{\wtilde{x^\star}}}{\rho_1}}+\frac{\rho_2}{\rho_1}\varphi\BPa{\Psi(\wtilde{x_0})-\Psi(\wtilde{x^\star})}<\rho, 
		\end{align*}
		which implies that $\xkp\in B(\xsol,\rho)$.
		
		If we start close enough to $\xsol$ so that \eqref{eq:local-conv-asump1}-\eqref{eq:local-conv-asump2} holds the sequence $\seq{\xk}$ will remain in the neighborhood $B(\xsol,\rho)$ and converges to a critical point, say $x^{\ast}$. Moreover $\Psi(\xk)\to\Phi(x^{\ast})\geq\Phi(\xsol)$. Let us assume that $\Phi(x^{\ast})>\Phi(\xsol)$. Since $\Phi$ has the KL property at $\xsol$ and thus
		\[
		\varphi'\Ppa{\Phi(x^\ast)-\Phi(\xsol)}\dist(0,\partial\Phi(x^\ast) \geq 1.
		\]
		This is a contradiction since $\varphi'(s)>0$ for $s\in]0,\eta[$ and $\dist(0,\partial\Phi(x^\ast))=0$ since $x^\ast$ is a critical point. We conclude that $x^\ast$ is indeed a global minimizer.
	\end{enumerate}
\end{proof}

\section{Proofs of Local Convergence}
At this juncture, we pause to present the following Lemma, which forms the essence of our framework. 
\begin{lemma}\label{lem:firm-nonexpansive}\textbf{(A Firmly nonexpansive map)} Let $G\in\Gamma_{0}(\bbR^n)$ and $\psi$ be a strongly convex function with full domain then the following map $J \eqdef\Ppa{\nabla\psi+\partial G}^{-1}$ and  $\Id-J$ are firmly non-expansive.  
\end{lemma}
Let us observe that if $\psi$ is not strongly convex but just convex, then from \eqref{eq:firm-nonexpensive} one has that $J$ is only monotone. 
\begin{proof}
	Let $x,y\in\bbR^n$ such that $p=J(x)$ and $q=J(y)$. Since $G\in\Gamma_{0}(\bbR^n)$ we have the following statements
	\[D^{-\nabla\psi(p)+x}_{G}(q,p)\geq0 \qandq D^{-\nabla\psi(q)+y}_{G}(p,q)\geq0.\]
	If we sum up both inequalities, we find that
	\begin{equation}\label{eq:firm-nonexpensive}
		D^{-\nabla\psi(p)+x}_{G}(q,p)+D^{-\nabla\psi(q)+y}_{G}(p,q)\geq0 \iff \pscal{Jx-Jy,x-y} \geq D_{\psi}(q,p)+D_{\psi}(p,q). 
	\end{equation}
	Therefore we get the desired result through strong convexity of $\psi$ and \cite[Proposition~4.4]{bauschke2011convex}. 
\end{proof}
\subsection{Proof of Lemma~\ref{lem:LemFinite}}\label{proof_finite}
\begin{proof}
	We have $0\in\partial\Phi(\xpa)$, since $G\in\PSF{\xpa}{\calM_{\xpa}}$ and $F\in C^1$ thanks to \cite[Corollay~4.7]{lewis_active_2002} (smooth perturbation of partly smooth functions), we have that $\Phi\in\PSF{\xpa}{\calM_{\xpa}}$. From the global convergence Theorem~\ref{thm:Global_conv_ABPG}, we have $\xk \xrightarrow{\Phi} \xpa$. Let us consider the iteration of Algorithm~\ref{alg:BPGBT} and define  
	\begin{align*}
		v_{k+1}=\frac{1}{\gak}\Ppa{-\nabla\psi(\xkp)+\nabla\psi(\yk)}+\nabla F(\xkp)-\nabla F(\yk),
	\end{align*}
	we have that $\forall k\in\bbN, v_{k}\in\partial \Phi(\xk)$. Besides,
	\begin{align*}
		\normm{v_{k+1}}\leq \frac{1}{\gak}\normm{\nabla \psi(\xkp)-\nabla\psi(\yk)}+\normm{\nabla F(\xkp)-\nabla F(\yk)}
	\end{align*}
	Since $\seq{\xk}$ is assumed to be bounded and so thus $\seq{\yk}$ by Lemma~\ref{lem:link-seq}, we deduce from Assumption~\ref{assump_B3} that there exists a positive scalar $M_1, M_2>0$ such that 
	\begin{align*}
		\normm{v_{k+1}}&\leq \Ppa{\frac{M_1}{\gak}+M_2}\normm{\xkp-\yk},\\
		&=\Ppa{\frac{M_1}{\gak}+M_2}\normm{\xkp-(1-\ak)(1-\akm)\xkm-(1-\ak)\akm\xk-\ak\xk},\\  
		&\leq \Ppa{\frac{M_1}{\gak}+M_2}\Ppa{\normm{\xkp-\xk}+(1-\ak)(1-\akm)\normm{\xk-\xkm}},\\
		&\leq \Ppa{\frac{M_1}{\gak}+M_2}\normm{\xkp-\xk}+ \Ppa{\frac{M_1}{\gak}+M_2}(1-\ak)(1-\akm)\normm{\xk-\xkm},  
	\end{align*}
	Therefore, we obtained that for $k\to \infty$ we get that $\ak\to a\in[\ua,\oa]$ implying that $\seq{\gak}$ converges. The latter combined with the fact that $\xk\to\xpa$, we deduce that $\normm{\vk}\to0$ as $k\to\infty$. We conclude that $\xk$ identifies $\calM_{\xpa}$ thanks to \cite[Proposition~10.12]{drusvyatskiy_optimality_2014}.
	
	\begin{enumerate}[label=(\roman*)]
		\item\label{pr-lemma-finite-i} If the active manifold $\calM_{\xsol}$ is an affine subspace, then $\calM_{\xsol} = \xsol + T_{\xsol}$ due to the normal sharpness property, and the claim follows immediately.
		
		\item\label{pr-lemma-finite-ii} When $G$ is locally polyhedral around $\xsol$, $\calM_{\xsol}$ is an affine subspace, and the identification of $\seq{\yk}$ and $\seq{\zk}$ follows from \ref{pr-lemma-finite-i}. For the rest, it is sufficient to observe that, by polyhedrality, for any $x$ in $\calM_{\xsol}$ near $\xsol$, $\partial G(x) = \partial G(\xsol)$, combining Fact~\ref{fct:normal_sharp} and Fact~\ref{fct:grad_psf}, we arrive at the second conclusion.
	\end{enumerate}
\end{proof}
The next Lemma gives the spectral properties of the matrices defined \eqref{eq:matrices-analysis}. 

\begin{lemma}\label{lem:spect_properysisty} Under  the Assumption~\ref{assump_A}-\ref{assump_B}, let $\xpa\in\gri{\Phi}\cap\nd{\Phi}$ such that $F$ is locally $C^2$ around $\xpa$. Then for any stepsize $\gamma\in\left]0,1/L\right]$,  we have
	\begin{enumerate}[label=(\roman*)]
		\item \label{lem:spect_property-i}   $H_{F}$ is symmetric positive definite  with eigenvalues in $\left]\gamma\sigma\sigma_{\psi}, L\Lma{\lambda}(\nabla^2\psi(\xpa))\gamma\right]$.
		\item \label{lem:spect_property-ii} $V$ has eigenvalues in 
		\[
		\left[\lambda_{\max}(\nabla^2\psi(\xpa))\Ppa{1-\gamma L},\lambda_{\max}(\nabla^2\psi(\xpa))-\gamma\sigma_{\psi}\sigma\right[,
		\]
		hence $H_{\psi}^{-1}V$ has eigenvalues in 
		\[\left[1-\gamma L\frac{\lambda_{\max}(\nabla^2\psi(\xpa))}{\sigma_{\psi}},1-\gamma\sigma\frac{\sigma_{\psi}}{\lambda_{\max}(\nabla^2\psi(\xpa))}\right[.
		\]
		\item\label{lem:spect_property-iii} $WH_{\psi}$ has eigenvalues in     $]0,q_\psi(\xpa)]$. 
		\item \label{lem:spect_property-iv} If either  $\xpa\in \nd{\Psi}$ or $\calM_{\xpa}$ is  affine then  $WV$ has    eigenvalues in $]-\Lambda,\Lambda[$ where we recall that 
		$\Lambda= \left|q_\psi(\xpa)-\gamma\sigma\right|$ with $q_{\psi}(\xpa)=\frac{\lambda_{\max}(\nabla^2\psi(\xpa))}{\sigma_{\psi}}$.
	\end{enumerate}
\end{lemma} 
\begin{proof}{\ }
	\begin{enumerate}[label=(\roman*)]
		\item Combining the fact that $F$ is locally $C^2$ and $\xpa\in\gri{\Psi}$ we get that $\exists\sig>0,\forall h\in T_{\xpa},$ 
		\begin{align*}
			\pscal{h;\nabla^2F(\xpa)h}\geq \pscal{h;\sig\nabla^2\psi(\xpa)h}\geq\sig\sig_{\psi}\normm{h}^2                                          
		\end{align*}
		where the last part comes from the strong convexity of $\psi$. This implies that $\Smi{\lambda}(H_{F})\geq \gamma\sig\sig_{\psi}$. 
		Since $F$ is $L-$smooth relative to $\psi$, $\forall h\in T_{\xpa},$ 
		\begin{align*}
			\pscal{h;\nabla^2F(\xpa)h}&\leq L\pscal{h;\nabla^2\psi(\xpa)h}\leq L\Lma{\lambda}(\nabla^2\psi(\xpa))\normm{h}^2.                         
		\end{align*}    
		
		\item We have that $V=H_{\psi}-H_F$ and that $H_\psi$ has eigenvalues in $[\sigma_{\psi},\lambda_{\max}(\nabla^2\psi(\xpa))]$. We combine this with the previous claim on $H_F$ and the eigenvalues of the difference of two positive definite matrices. 
		
		For the second  claim, let us observe that $H_{\psi}^{-1}V$ is similar to the following matrice $\Id-H_{\psi}^{-1/2}H_FH_{\psi}^{-1/2}$. Besides, we have that $H_{\psi}^{-1/2}H_FH_{\psi}^{-1/2}$ has eigenvalues in 
		$
		\left[\gamma\sigma\frac{\sigma_{\psi}}{\lambda_{\max}(\nabla^2\psi(\xpa))},\gamma L\frac{\lambda_{\max}(\nabla^2\psi(\xpa))}{\sigma_{\psi}}\right[
		$ thus, the difference of two positive definite matrices.
		
		\item  From Remark~\ref{rmk:positive-U}, we have that $W$ has eigenvalues in  $\left]0,1/\sigma_{\psi}\right]$ which implies  that $WH_\psi$ has eigenvalues in $\left]0,\frac{\lambda_{\max}(\nabla^2\psi(\xpa))}{\sigma_{\psi}}\right]$. 
		
		\item       Let us observe that $WV=W^{1/2}\Ppa{W^{1/2}VW^{1/2}}W^{-1/2}$, thus  $WV$ is  similar to $W^{1/2}VW^{1/2}$. We have 
		\begin{align*}
			\normm{W^{1/2}VW^{1/2}}&\leq \normm{W^{1/2}}\normm{V}\normm{W^{1/2}}\leq\left|\lambda_{\max}(\nabla^2\psi(\xpa))-\gamma\sigma_{\psi}\sigma\right| \normm{W^{1/2}}^2.\\ 
			&=\left|\frac{\lambda_{\max}(\nabla^2\psi(\xpa))}{\sigma_{\psi}}-\gamma\sigma\right|, 
		\end{align*}
		where we used the fact that either $\xpa\in \nd{\Psi}$ or $\calM_{\xpa}$ is affine holds then from Remark~\ref{rmk:positive-U} we have $\normm{W}\leq\frac{1}{\sigma_{\psi}}$. 
	\end{enumerate}
\end{proof}

Let  us define the following matrices \ie, 
\begin{align}\label{eq:matrices-dependent}
	&H^k_{F}\eqdef\gak\proj{T_{\xpa}}\nabla^2F(\xpa)\proj{T_{\xpa}},\quad V^k\eqdef H_{\psi} - H^k_{F},&U^k\eqdef\gak \nabla^2_{\calM_{\xpa}}\Phi(\xpa)\proj{T_{\xpa}}-H_F.
\end{align}
To enhance readability, we introduce simplified notation for any $k\in\bbN$,  
\begin{flalign}\label{eq:notation-matr}
	&b_k=(1-\ak)\akm+\ak \qandq c_k=(1-\ak)(1-\akm),\\
	&b=2a-a^2\qandq c=(1-a)^2.
\end{flalign}
From this notation, we have this obvious Lemma. 
\begin{lemma}\label{lem:coeff-seq-conv}Let us  consider the sequences define in \eqref{eq:notation-matr}, If $\ak\to a$ as $k\to\infty$ then we have $b_k\to b$ and $c_k\to c$.    
\end{lemma}
Following the work of \cite[Section~B]{liang_activity_2017}, we also define the matrices
\begin{flalign*}
	& M^k_1=\begin{bmatrix}
		bW(V^k-V),&cW(V^k-V)
	\end{bmatrix},\quad M^k_2=\begin{bmatrix}
		\pa{b_k-b}WV^k,&(c_k-c)WV^k
	\end{bmatrix}. 
\end{flalign*}

Therefore, we have the following proposition. 
\begin{proposition}\label{pro:estimation} Under the same assumptions as Proposition~\ref{pro:local-linearization}, for $k$ large enough we have  
	\begin{align}
		& \normm{\yk-\xpa}=O(\normm{\dk}),\normm{\rkp}=O(\normm{\dk}), \normm{\xkp-\yk}=O(\normm{\dk}),\label{eq:boundness-line1}\\
		&\normm{\nabla F(\yk) -F(\xkp)}=O\Ppa{\normm{\dk}},\quad\normm{\nabla\psi(\yk)-\nabla\psi(\xkp)}=O\Ppa{\normm{\dk}}, \label{eq:boundness-line2} 
	\end{align}
	and 
	\begin{equation}\label{eq:boundness-line3}
		\normm{W(U^k-U)(\xkp-\xpa)}=o(\normm{\dk}), \normm{M_1^k\dk}=o\pa{\normm{\dk}}, \normm{M_2^k\dk}=o\pa{\normm{\dk}}.   
	\end{equation}
\end{proposition}
\begin{proof}$~$
	\begin{itemize}
		\item From the definition of the sequences we have, 
		\begin{flalign*}
			\normm{\yk-\xpa}=&\normm{(1-\ak)\zk+\ak\xk-\xpa},\\
			=&\normm{(1-\ak)(zk-\xpa)+\ak(\xk-\xpa)},\\
			=&\normm{(1-\ak)(1-\akm)\rkm+\Ppa{(1-\ak)\akm+\ak}\rk},\\
			\leq& (1-\ak)(1-\akm)\Ppa{\normm{\rkm}+\normm{\rk}},\\
			\leq&\sqrt{2}\normm{\dk}. 	
		\end{flalign*} 
		
		\item We recall that $\rkp=\xkp-\xpa$ thus,
		\begin{flalign*}\label{eq:your_equation_label}
			\begin{split}
				\normm{\rkp}&=\normm{\pa{\nabla\psi+\gak\partial G}^{-1}\pa{\nabla\psi(\yk)-\gak\nabla F(\yk)}-\pa{\nabla\psi+\gak\partial G}^{-1}\pa{\nabla\psi(\xpa)-\gak\nabla F(\xpa)}},\\
				&\leq \normm{\nabla\psi(\yk)-\nabla\psi(\xpa)}+\gak\normm{\nabla F(\yk)-\nabla F(\xpa)},\\
				&\leq \Ppa{M_{1}+M_2\gak}\normm{\yk-\xpa},\\
				&\leq \Ppa{M_{1}+M_2\gak}\sqrt{2}\normm{\dk},
			\end{split}
		\end{flalign*}
		where we used the non-expansiveness of the mapping $(\nabla\psi+\gak\partial G)^{-1}$ (see Lemma~\ref{lem:firm-nonexpansive}), the boundedness of the sequence and again Assumption~\ref{assump_A3}.
		
		\item From Assumption~\ref{assump_A3}, there exists $M_1>0$ large enough such that
		\begin{flalign*}
			\normm{\nabla\psi(\yk)-\nabla\psi(\xkp)}\leq M_1 \normm{\xkp-\yk}=O(\normm{\dk}).  
		\end{flalign*}
		Similarly, there exists $M_2$ large enough such that
		\begin{align*}
			\normm{\nabla F(\yk)-\nabla F(\xkp)}\leq M_2\normm{\yk-\xkp}\leq M_2\sqrt{2}\normm{\dk}=O(\normm{\dk}). 
		\end{align*}

		\item Let us now turn to the proof of \eqref{eq:boundness-line3}, from the definition of $\Phi$, we have that
		\begin{flalign}
			& \lim_{k \to \infty}\frac{\normm{W(U^k - U)\rkp}}{\normm{\rkp}}\leq \lim_{k \to \infty} |\gak - \gamma| \normm{W}\normm{\nabla_{\calM_{\xpa}}\Phi(\xpa)\proj{T_{\xpa}}}=0,
		\end{flalign}
		since $\gak\to\gamma$ which means that $\normm{W(U^k - U)\rkp}=o(\normm{\rkp})=o(\normm{\dk})$, where we have used \eqref{eq:boundness-line1}.  
		
		\item We have that,  
		\begin{align*}
			\lim_{k \to \infty} \frac{\normm{M^k_1\dk}}{\normm{\dk}}&=\lim_{k\to\infty}\frac{\normm{ bW(V^k-V)\rk+cW(V^k-V)\rkm}}{\normm{\dk}} \\
			&\leq\lim_{k\to\infty}\frac{\max\pa{\abs{b},\abs{c}}\normm{W}\normm{V^k-V}(\normm{\rk}+\normm{\rkp})}{\normm{\dk}} \\
			&\leq \lim_{k \to \infty}\frac{\max\pa{\abs{b},\abs{c}}\normm{W}|\gak-\gamma|\normm{\proj{T_{\xpa}}\nabla^2F(\xpa)\proj{T_{\xpa}}}\sqrt{2}\normm{\dk}}{\normm{\dk}} \\
			&=\lim_{k\to\infty}\sqrt{2}|\gak-\gamma|\max\pa{\abs{b},\abs{c}}\normm{W}\normm{\proj{T_{\xpa}}\nabla^2F(\xpa)\proj{T_{\xpa}}}=0,  
		\end{align*}
		as the term $\max\pa{\abs{b},\abs{c}}\normm{W}\normm{\proj{T_{\xpa}}\nabla^2F(\xpa)\proj{T_{\xpa}}}$ is bounded  since $\normm{W}\normm{\proj{T_{\xpa}}\nabla^2F(\xpa)\proj{T_{\xpa}}}\leq L\lambda_{\max}(\nabla^2\psi(\xpa))^{-1}\normm{H_{\psi}}\leq1$ and $\max\pa{\abs{b},\abs{c}}\leq 1$.

		\item To finish the proof, we have
		\begin{align*}
			\lim_{k \to \infty} \frac{\normm{M^k_2\dk}}{\normm{\dk}}&=\lim_{k\to\infty}\frac{\normm{\pa{b_k-b}WV^k\rk+(c_k-c)WV^k\rkm}}{\normm{\dk}},\\
			&\leq\lim_{k\to\infty}\frac{\max\Ppa{\abs{b_k-b},\abs{c_k-c}}\normm{W}\normm{V^k}\sqrt{2}\normm{\dk}}{\normm{\dk}},\\
			&\leq\lim_{k\to\infty} \max\Ppa{\abs{b_k-b},\abs{c_k-c}}\normm{W}\normm{V^k}\sqrt{2}=0,
		\end{align*}
		where  we use \eqref{eq:boundness-line1}, the fact that $\normm{V^k}\to\normm{V}<\infty$,  the  boundedness of $\normm{W}$ combined with Lemma~\ref{lem:coeff-seq-conv} to get the result.
	\end{itemize}
	
\end{proof}

\subsection{Proof of Proposition~\ref{pro:local-linearization}}\label{proof:local-linearization}
\begin{proof}$~$
	\newline 
	Since  $\seq{\zk}$ is a sequence generated by Algorithm~\ref{alg:BPGBT} converging to $\xpa\in\crit{\Phi}$. From the finite identification Lemma~\ref{lem:LemFinite},  there exists  $K\in\bbN$ such that $\xk$ is close enough to $\xpa$ for $k\geq K$. Let $T_{\xkp},T_{\xpa}$ be their corresponding tangent spaces, and define $\tau_{k+1}:T_{\xpa}\rightarrow T_{\xkp}$ the parallel translation along the unique geodesic joining $\xkp$ to $\xpa$. From the definition of the point $\xkp$ and the  fact that $\xpa\in\crit{\Phi}$ we have 
	\begin{eqnarray*}
		\nabla\psi(\yk)-\nabla\psi(\xkp)-\gak\Ppa{\nabla F(\yk)-\nabla F(\xkp)}\in\gak\partial\Phi(\xkp)&\\
		0\in\gak\partial\Phi(\xpa).&
	\end{eqnarray*}
	We project now this inclusions over the tangents spaces $T_{\xkp}$ and $T_{\xpa}$ respectively  to get that
	\begin{eqnarray*}
		\tau_{k+1}^{-1}\proj{T_{\xkp}}\BPa{\nabla\psi(\yk)-\nabla\psi(\xkp)-\gak\Ppa{\nabla F(\yk)-\nabla F(\xkp)}}=&\gak \tau^{-1}_{k+1}\nabla_{\calM_{\xpa}}\Phi(\xkp)\\
		0=&\gak\nabla_{\calM_{\xpa}}\Phi(\xpa)
	\end{eqnarray*}
	where we have used the Fact~\ref{fct:grad_psf}. After summing both lines  and subtracting the value $\tau_{k+1}^{-1}\proj{T_{\xkp}}\nabla\psi(\xpa)$ we have
	\begin{equation}\label{eq:projTxpa}
		\begin{aligned}
			&\gak\Ppa{\tau^{-1}_{k+1}\nabla_{\calM_{\xpa}}\Phi(\xkp) - \nabla_{\calM_{\xpa}}\Phi(\xpa)} \\
			&= \tau_{k+1}^{-1}\proj{T_{\xkp}}\Ppa{\nabla\psi(\yk) - \nabla\psi(\xkp)} - \gak\tau_{k+1}^{-1}\proj{T_{\xkp}}\Ppa{\nabla F(\yk) - \nabla F(\xkp)}.
		\end{aligned}
	\end{equation}
	We  combined Lemma~\ref{lem: Tgtxk-to-Tgtx}, \eqref{eq:boundness-line2} and Lemma~\ref{lem: Riem-Tayl} due to the local $C^2$-smoothness of $\psi$  to obtain that
	\begin{flalign*}\label{eq:psi-xkp-xpa}
		& \tau_{k+1}^{-1}\proj{T_{\xkp}}\Ppa{\nabla\psi(\yk)-\nabla\psi(\xkp)} \\
		&= \proj{T_{\xpa}}\Ppa{\nabla\psi(\yk)-\nabla\psi(\xkp)}+o(\normm{\xkp-\yk}), \\
		&= \proj{T_{\xpa}}\Ppa{\nabla\psi(\yk)-\nabla\psi(\xpa)}-\proj{T_{\xpa}}\Ppa{\nabla\psi(\xkp)-\nabla\psi(\xpa)}+o(\normm{\dk}),\\
		&=\proj{T_{\xpa}}\nabla^2\psi(\xpa)(\yk-\xpa)+o(\normm{\yk-\xpa})-\proj{T_{\xpa}}\nabla^2\psi(\xpa)(\xkp-\xpa)+o(\normm{\rkp})+o(\normm{\dk}),\\
		&= \proj{T_{\xpa}}\nabla^2\psi(\xpa)(\yk-\xpa)
		-\proj{T_{\xpa}}\nabla^2\psi(\xpa)(\xkp-\xpa)+o(\normm{\dk}),\\
		&= \proj{T_{\xpa}}\nabla^2\psi(\xpa)\proj{T_{\xpa}}(\yk-\xpa)
		-\proj{T_{\xpa}}\nabla^2\psi(\xpa)\proj{T_{\xpa}}(\xkp-\xpa)+o(\normm{\dk}),  \numberthis
	\end{flalign*}
	where we have used \eqref{eq:boundness-line2}, \cite[Lemma~5.1]{liang_activity_2015} and the fact that $r_{k+1}=O(\normm{\dk})$.
	Using similar arguments, we have
	\begin{flalign*}\label{eq:gradf-yk-xkp}
		& \tau_{k+1}^{-1}\proj{T_{\xkp}}\Ppa{\nabla F(\yk)-\nabla F(\xkp)} \\
		&= \proj{T_{\xpa}}\Ppa{\nabla F(\yk)-\nabla F(\xkp)}+o(\normm{\dk}), \\
		&= \proj{T_{\xpa}}\Ppa{\nabla F(\yk)-\nabla F(\xpa)}-\proj{T_{\xpa}}\Ppa{\nabla F(\xkp)-\nabla\psi(\xpa)}+o(\normm{\dk}),\\
		&=\proj{T_{\xpa}}\nabla^2F(\xpa)(\yk-\xpa)+o(\normm{\yk-\xpa})-\proj{T_{\xpa}}\nabla^2F(\xpa)(\xkp-\xpa)+o(\normm{\xkp-\xpa})+o(\normm{\dk}),\\
		&= \proj{T_{\xpa}}\nabla^2F(\xpa)(\yk-\xpa)
		-\proj{T_{\xpa}}\nabla^2F(\xpa)(\xkp-\xpa)+o(\normm{\dk}),\\
		&= \proj{T_{\xpa}}\nabla^2F(\xpa)\proj{T_{\xpa}}(\yk-\xpa)
		-\proj{T_{\xpa}}\nabla^2F(\xpa)\proj{T_{\xpa}}(\xkp-\xpa)+o(\normm{\dk}).
		\numberthis
	\end{flalign*}
	Moreover, we have
	\begin{equation}\label{eq:taylor_Psi}
		\tau^{-1}_{k+1}\nabla_{\calM_{\xpa}}\Phi(\xkp)-\nabla_{\calM_{\xpa}}\Phi(\xpa)=\nabla^2_{\calM_{\xpa}}\Phi(\xpa)\proj{T_{\xpa}}\Ppa{\xkp-\xpa} +o(\normm{\dk}).
	\end{equation}
	We replace now the expressions \eqref{eq:psi-xkp-xpa}, \eqref{eq:gradf-yk-xkp}, \eqref{eq:taylor_Psi}, in \eqref{eq:projTxpa}, we obtain
	\begin{equation}\label{eq:replacing-result}
		\begin{aligned}
			&\BPa{\proj{T_{\xpa}}\nabla^2\psi(\xpa)\proj{T_{\xpa}}+\gak\nabla^2_{\calM_{\xpa}}\Phi(\xpa)\proj{T_{\xpa}}-\gak\proj{T_{\xpa}}\nabla^2F(\xpa)\proj{T_{\xpa}}}\pa{\xkp-\xpa}= \pa{H_{\psi}+U^k}\pa{\xkp-\xpa}\\
			&=\proj{T_{\xpa}}\nabla^2\psi(\xpa)\proj{T_{\xpa}}(\yk-\xpa) -\gak\proj{T_{\xpa}}\nabla^2F(\xpa)\proj{T_{\xpa}}(\yk-\xpa)+o(\normm{\dk}). \\
		\end{aligned}
	\end{equation}
	We get by factorizing and replacing \eqref{eq:matrices-dependent} that
	\begin{equation}\label{eq:replacing-result-2}
		\begin{aligned}
			\pa{H_{\psi}+U^k}\pa{\xkp-\xpa}=V^k(\yk-\xpa)+o(\normm{\dk}).
		\end{aligned}
	\end{equation}
	Now we replace the expressions of the  inertial term in term of $\xk, \xkm$ to obtain that 
	\begin{equation}\label{eq:replacing-result-i}
		\pa{H_{\psi}+U^k}\rkp=\Ppa{(1-\ak)\akm+\ak}V^k\rk+(1-\ak)(1-\akm)V^k\rkm+o(\normm{\dk}),
	\end{equation}
	We can write further  
	\begin{align*}
		\Ppa{H_{\psi}+U}\rkp&=\Ppa{U-U^k}\rkp+\Ppa{(1-\ak)\akm+\ak}V^k\rk+(1-\ak)(1-\akm)V^k\rkm\\
		&\quad+o(\normm{\dk}),
	\end{align*}
	Thanks to Remark~\ref{rmk:positive-U}, it is possible to invert the matrice $U+H_{\psi}$ we have, 
	\begin{align*}
		\rkp&=W\Ppa{U-U^k}\rkp+\Ppa{(1-\ak)\akm+\ak}WV^k\rk+(1-\ak)(1-\akm)WV^k\rkm\\
		&\quad+o(\normm{W\dk}),
	\end{align*}
	Let us use the notation \eqref{eq:notation-matr}   with the estimates \eqref{eq:boundness-line3}
	
	\begin{equation*}
		\dkp=\Ppa{M+\begin{bmatrix}
				M^k_1\\
				0
			\end{bmatrix}+\begin{bmatrix}
				M^k_2\\
				0
		\end{bmatrix}}d_k+o(\normm{\dk})=M\dk+o(\normm{\dk}), 
	\end{equation*}
	which concludes the proof. 
\end{proof}
\subsection{Proof of Proposition~\ref{pro:poly-spect-anal}}\label{pr:pro:poly-spect-anal}
\begin{proof}
	We have that 
	\begin{align*}
		M\begin{pmatrix}r_1\\r_2\end{pmatrix}&=\begin{bmatrix}(2a-a^2)H_{\psi}^{-1}V&\pa{a-1}^2H_{\psi}^{-1}V\\
			\Id&0   \end{bmatrix}\begin{pmatrix}r_1\\r_2\end{pmatrix}\\
		&=\begin{pmatrix}(2a-a^2)H_{\psi}^{-1}Vr_1+\pa{a-1}^2H_{\psi}^{-1}Vr_2\\r_1 \end{pmatrix}=\varrho\begin{pmatrix}r_1\\r_2\end{pmatrix}, 
	\end{align*}
	therefore $r_1=\varrho r_2$, and insert it in the first identity to get that
	\begin{align*}
		\varrho^2 r_2=(2a-a^2)H_{\psi}^{-1}V\varrho r_2+(1-a)^2H_{\psi}^{-1}Vr_2
	\end{align*}
	which means that 
	\[
	\Ppa{(2a-a^2)\varrho+(1-a)^2}H_{\psi}Vr_2=\varrho^2r_2
	\] 
	thus there exists $\eta$ such that $ H_{\psi}^{-1}Vr_2=\eta r_2$ moreover $\eta$ satisfies the following equation 
	\begin{align}\label{eq:polynomial-eq}
		\varrho^2-(2a-a^2)\varrho\eta-(1-a)^2\eta=0.
	\end{align}
	The rest of the proof follows  exactly the same step as the proof of \cite[Proposition~4.7,Corollary~4.9]{liang_activity_2017} and we conclude with Lemma~\ref{lem:spect_properysisty}-\ref{lem:spect_property-iv}.
\end{proof}

\section{Proof of the Escape Property}
\subsection{Proof of Theorem~\ref{thm: iga_glbl_escp}}\label{pr:thm: iga_glbl_escp}
Let us define the following mapping for $a\in[\ua,\oa]$. 
\begin{align}\label{eq:operator_T}
	\mathbf{T}(x_2,x_1)\eqdef\begin{bmatrix}
		\nabla\psi^{-1}\Ppa{\nabla\psi\Pa{y(x_2,x_1)}-\gamma\nabla F(y(x_2,x_1))}\\
		x_2
	\end{bmatrix},
\end{align}
where
\[
y(x_2,x_1)=\pa{2a-a^2}x_2+ \pa{1-a}^2x_1.
\]
It is simple to see that $\xpa\in\crit{\Phi}$ if and only if $(\xpa,\xpa)$ is a fixed point of the operator $\mathbf{T}$. We have the following lemma which is an extension of the \cite[Lemma~A.2]{godeme2023provable} to the inertial case.
\begin{lemma}\label{lem:escape-mdinertiel} Let $\mathbf{T}$ be defined as in \eqref{eq:operator_T} then, 
	\begin{enumerate}
		\item[(a)] For all $(x_2,x_1)\in\bbR^{2n}, \det D\mathbf{T}(x_2,x_1)\neq0$, 
		\item[(b)] The set of strict saddle points is contained in the following set 
		\begin{flalign}
			U_{\mathbf{T}}&\eqdef\left\{(x_2,x_1)\in \bbR^{2n}: \mathbf{T}(x_2,x_1)=\begin{bmatrix}x_2\\x_1\end{bmatrix},\max_i\abs{\lambda_i(D\mathbf{T}(x_1,x_2))}>1\right\},\\
			&=\left\{(x,x)\in \bbR^{2n}: \mathbf{T}(x,x)=\begin{bmatrix}x\\x\end{bmatrix},\max_i\abs{\lambda_i(D\mathbf{T}(x,x))}>1\right\}.
		\end{flalign}
	\end{enumerate}
\end{lemma}
\begin{proof}$~$
	\begin{enumerate}
		\item[(a)] Since $\psi$ is a $C^2$ function, and thus $\nabla\psi$ is $C^{1}$, and as $\psi$ is strongly convex, the inverse function theorem ensures that $(\nabla\psi)^{-1}$ is a local diffeomorphism. Moreover, $F$ is $C^2$, therefore we define the following for simplicity
		\begin{multline*}
			\mathbf{A}(x_2,x_1)\eqdef (2a-a^2)\nabla^2\psi^{-1}\Ppa{\nabla\psi(y\pa{x_2,x_1})-\gamma\nabla F\pa{y\pa{x_1,x_2}}}\\\Ppa{\nabla^2\psi(y\pa{x_2,x_1})-\gamma\nabla^2F(y\pa{x_2,y_1})},   
		\end{multline*}
		and 
		\begin{multline*}
			\mathbf{B}(x_2,x_1)\eqdef (1-a)^2\nabla^2\psi^{-1}\Ppa{\nabla\psi(y\pa{x_2,x_1})-\gamma\nabla F\pa{y\pa{x_1,x_2}}}\\\Ppa{\nabla^2\psi(y\pa{x_2,x_1})-\gamma\nabla^2F(y\pa{x_2,y_1})},   
		\end{multline*}
		then one can write that   
		\begin{flalign}
			D\mathbf{T}(x_2,x_1)=\begin{bmatrix}
				\mathbf{A}(x_2,x_1)&\mathbf{B}(x_2,x_1)\\
				\Id&0
			\end{bmatrix}.    
		\end{flalign}
		We deduce that for any $(x_2,x_1)$ we have that $\det D\mathbf{T}(x_2,x_1)=\det (-\mathbf{B}(x_2,x_1)).$
		Therefore, it suffices to show that  $\det (-\mathbf{B}(x_2,x_1))\neq0$ which hold  since   $\forall \pa{x_2,x_1},\quad\nabla^2\psi(y\pa{x_2,x_1})-\gamma \nabla^2F(y\pa{x_2,x_1})$ is invertible. Indeed we have,
		\begin{align*}
			\nabla^2\psi(y\pa{x_2,x_1})-\gamma \nabla^2F(y\pa{x_2,x_1})\succ(1-\gamma L) \nabla^2\psi(y\pa{x_2,x_1})\succ 0,
		\end{align*}
		where we have used the $L-$smooth adaptable property and the strong convexity of $\psi$.  
		\item[(b)]  Let $\xpa$ be a strict saddle point therefore $\pa{\xpa,\xpa}$ is a fixed point of $\mathbf{T}$. To have that $(\xpa,\xpa)\in U_{\mathbf{T}}$ it remains to show that $D\mathbf{T}$ has an eigenvalue of magnitude greater than 1. 
		We have
		\begin{align*}
			D\mathbf{T}(\xpa,\xpa)=\begin{bmatrix}
				(2a-a^2)\Ppa{\Id-\gamma\nabla^2\psi\pa{\xpa}^{-1}\nabla^2F(\xpa)},&(1-a)^2\Ppa{\Id-\gamma \nabla^2\psi\pa{\xpa}^{-1}\nabla^2F(\xpa)}\\
				\Id&0
			\end{bmatrix}.
		\end{align*}
		Let us remark that $D\mathbf{T}(\xpa,\xpa)$ is similar to the following matrix
		\begin{align*}
			\wtilde{\mathbf{T}}=\begin{bmatrix}
				(2a-a^2)\Ppa{\Id-\gamma H_{\psi}^{-1/2}\nabla^2F(\xpa)H_{\psi}^{1/2}},&(1-a)^2\Ppa{\Id-\gamma H_{\psi}^{-1/2}\nabla^2F(\xpa)H_{\psi}^{1/2}}\\
				\Id&0
			\end{bmatrix},
		\end{align*}
		where we have applied the  following transformation $\begin{bmatrix}H_{\psi}^{-1/2}&0\\0&H_{\psi}^{-1/2}
		\end{bmatrix}D\mathbf{T}(\xpa,\xpa)\begin{bmatrix}H_{\psi}^{1/2}&0\\0&H_{\psi}^{1/2}
		\end{bmatrix}$. 
		For $\gamma<1/L$, the symmetric matrix $\Id-\gamma H_{\psi}^{-1/2}\nabla^2F(\xpa)H_{\psi}^{1/2}$ has an eigenvalue of magnitude greater than one,(see Lemma~\ref{lem:spect_properysisty}). Let us denote by
		$v$ the eigenvector associated with this eigenvalue that we denote $\eta>1.$ We claim that the vector $\begin{bmatrix} v\\0\end{bmatrix}$ is an eigenvector associated with the eigenvalue $\eta$ of the  matrix $\wtilde{\mathbf{T}}$.
		Indeed, we have 
		\begin{align*}
			\wtilde{\mathbf{T}}\begin{bmatrix} v\\0\end{bmatrix}=\begin{bmatrix}(2a-a^2)v\eta+(1-a)^2v\eta  \\0\end{bmatrix}=\eta\begin{bmatrix}v\\0\end{bmatrix}, 
		\end{align*}
		which conclude the proof. 
	\end{enumerate}
\end{proof}
To show our Claim, we combine the global convergence result of Theorem~\ref{thm:Global_conv_ABPG}-\ref{thm:Global_conv_ABPG-i}  the previous Lemma and the center stable manifold theorem see \cite[Corollary 1]{lee_first-order_2019}. 
This allows us to say that  the set \[\Bba{(z_0,z_{-1})\in\bbR^{2n}:\lim\limits_{k\to\infty} \mathbf{T}^{k}((z_0,z_{-1}))\in\text{strisad}(\Phi)}\] has measure zero.   

\subsection{Proof of Lemma~\ref{lem:unstablebpg}}\label{pr:lem:unstablebpg}
Let us first observe that since $\xpa\in\nd{\Phi}$,  using the same arguments  as for the finite time identification Lemma~\ref{lem:LemFinite} for any $x=\pa{x_2,x_1}\in\bbR^{2n}$ near $\pa{\xpa,\xpa}$, $\mathbf{T}\pa{x_2,x_1}\in\wtilde{\calW}_{\xpa}$.
The subsequent portion of the proof faithfully adheres to the perturbation analysis used in \cite[Theorem~4.1]{davis_proximal_2022} with a minor adjustment to the inertial Bregman case. 

Let $\widetilde{G}:\bbR^n\rightarrow\bbR$ be any $C^2-$smooth  extension (representative) of $G$ on the neighborhood of $\xpa$ in $\calW_{\xpa},$ consider the following problem defining  $P_1\mathbf{T}$ near $\pa{\xpa,\xpa}$,
\begin{equation}\tag{$\widetilde{\calP}_x$}
	\min\limits_{u\in\calW_{\xpa}}\Upsilon(x,u)\eqdef \Bba{F(y(x))+\widetilde{G}(u)+\pscal{\nabla F(y(x)),u-y(x)}+\frac{1}{\gamma}D_{\psi}(u,y(x))}. 
\end{equation}
We can write the map $\mathbf{T}$ as
\[
\mathbf{T}(x)=\begin{bmatrix}
	\min\limits_{u\in\calW_{\xpa}}\Upsilon(x,u)\\
	x
\end{bmatrix},
\]
To apply the perturbation result \cite[Theorem 3.1]{shapiro_second_1985} to $\Upsilon$. We need a quadratic growth condition which we got using the fact that  $\psi$ is strongly convex and thus $u\mapsto\Upsilon(x,u)$ is also $\frac{\sigma_{\psi}}{\gamma}$-strongly convex. We also need to check a  level-boundedness condition. This comes from a sufficient condition see \cite[Lemma~2.4]{davis_proximal_2022}. $u(x)$, the minimizer of $\Upsilon(z,\cdot)$ is a continuous map in near $(\xpa,\xpa)$ then we apply the perturbation analysis \cite[Theorem 3.1]{shapiro_second_1985} with the following Lagrangian function
\begin{equation*}
	\calL(x,u,\lambda)\eqdef \Upsilon(x,u)+\pscal{W(u),\lambda)},
\end{equation*}
where $\lambda$ is the vector of Lagrange multipliers.

Since $u\longmapsto\Upsilon(x,u)$ is a strongly convex function and $u\pa{\xpa,\xpa}=\xpa$ is the unique minimizer on $\calW_{\xpa}.$ Thus there exists optimal Lagrange multipliers $\bar{\lambda}$ such that 
\begin{equation*}
	\nabla_u\calL\Ppa{(\xpa,\xpa),\xpa,\bar{\lambda}}=\nabla\wtilde{G}(\xpa)+\nabla F(\xpa)+\sum_{i=1}^{n}\bar{\lambda}_{i}\nabla W_i(\xpa)=0.
\end{equation*}
From \cite[Theorem~3.1]{shapiro_second_1985}, we get that locally for any $x$ near $(\xpa,\xpa), u(x)=P_1\mathbf{T}(x)$ is a $C^1-$ smooth map and  we deduce that for any  $h=(h_1,h_2)\in\bbR^{2n}$ with $\normm{h}\rightarrow0$,   
\begin{equation}\label{dPsi}
	\pscal{\nabla P_1\mathbf{T}(\xpa,\xpa),h}=\argmin_{v\in\Tt_{\calW_{\xpa}}}  2\pscal{\nabla^2_{zu}\calL\Ppa{(\xpa,\xpa),\xpa,\bar{\lambda}}v,h}+\pscal{\nabla^2_{uu}\calL\Ppa{(\xpa,\xpa),\xpa,\bar{\lambda}}v,v}.
\end{equation}
The expression of the Hessians of the Lagrangian is of the form $2n\times n$ 
\begin{align*}
	\calH_{zu}&\eqdef\nabla^2_{zu}\calL\Ppa{(\xpa,\xpa),\xpa,\bar{\lambda}}\\
	&=\begin{bmatrix}
		(2-a)\Ppa{\nabla^2F(\xpa)-\frac{1}{\gamma}\nabla^2\psi(\xpa)},\\
		(1-a)^2\Ppa{\nabla^2F(\xpa)-\frac{1}{\gamma}\nabla^2\psi(\xpa)}
	\end{bmatrix},      
\end{align*}
and 
\[
\calH_{uu}\eqdef\nabla^2_{uu}\calL\Ppa{(\xpa,\xpa),\xpa,\bar{\lambda}}=\nabla^2\wtilde{G}(\xpa)+\frac{1}{\gamma}\nabla^2\psi(\xpa)+\sum_{i=1}^{n}\bar{\lambda}_{i}\nabla^2_{uu}W_i(\xpa).
\]
The minimization problem \eqref{dPsi} is a quadratic form over the tangent space $\Tt_{\calW_{\xpa}}$ thus the minimizer $\bar{v}$ must satisfy the following condition 
\begin{equation}\label{eq:vbar-sol}
	\proj{\Tt_{\calW_{\xpa}}}\calH_{uu}\proj{\Tt_{\calW_{\xpa}}}\bar{v}+\proj{\Tt_{\calW_{\xpa}}}P_1\calH_{zu}^{\top}\proj{\Tt_{\calW_{\xpa}}}h_1+\proj{\Tt_{\calW_{\xpa}}}P_2\calH_{zu}^{\top}\proj{\Tt_{\calW_{\xpa}}}h_2=0.
\end{equation}
Let us denote  $\widetilde{\calH}_{uu}=\proj{\Tt_{\calW_{\xpa}}}\calH_{uu}\proj{\Tt_{\calW_{\xpa}}},\wtilde{\calH}_{zu}^2=\proj{\Tt_{\calW_{\xpa}}}P_1\calH_{zu}^{\top}\proj{\Tt_{\calW_{\xpa}}}$ and $\wtilde{\calH}_{zu}^1=\proj{\Tt_{\calW_{\xpa}}}P_2\calH_{zu}^{\top}\proj{\Tt_{\calW_{\xpa}}}$. \eqref{eq:vbar-sol} becomes
\[
\wtilde{\calH}_{uu}\bar{v}+ \wtilde{\calH}^2_{zu}h_1+\wtilde{\calH}^1_{zu}h_2=0. 
\]
Since $\xpa$ is  the unique minimizer of $\Upsilon\Ppa{(\xpa,\xpa),\cdot}$ and we also observe that
$\Upsilon\Ppa{(\xpa,\xpa),\cdot}$ has the same active manifold $\calW_{\xpa}$ as a sum of $G$ and a smooth function this implies that
$\Upsilon\Ppa{(\xpa,\xpa),\cdot}$ has at least a quadratic growth near $\xpa$. Therefore  $\wtilde{\calH}_{uu}$ is symmetric positive definite and invertible. Then we solve \eqref{dPsi} to get that
\begin{equation*}
	\pscal{\nabla P_1\mathbf{T}\Ppa{(\xpa,\xpa)},h}=-\wtilde{\calH}_{uu}^{-1}\wtilde{\calH}^2_{zu}h_1-\wtilde{\calH}_{uu}^{-1}\wtilde{\calH}^1_{zu}h_2, \quad \forall h_1,h_2\in \Tt_{\calW_{\xpa}}.
\end{equation*}
At this point, we  get that $\mathbf{T}$ is a $C^1-$smooth map. We immediately deduce that 
\begin{equation*}
	\pscal{\nabla \mathbf{T}\Ppa{(\xpa,\xpa)},h}=\begin{bmatrix}
		-\wtilde{\calH}_{uu}^{-1}\wtilde{\calH}^2_{zu}h_2-\wtilde{\calH}_{uu}^{-1}\wtilde{\calH}^1_{zu}h_1\\
		h_1
	\end{bmatrix}, \quad \forall h_1,h_2\in \Tt_{\calW_{\xpa}}.
\end{equation*}

It remains now to show that $\nabla \mathbf{T}\pa{\xpa,\xpa}$ has at least one eigenvalue greater than one. 
Let us first show that when $\xpa\in\actstrisad{\Phi}$, $-\wtilde{\calH}_{uu}^{-1}\wtilde{\calH}^2_{zu}$ has an eigenvalue greater than one. Let $\eta$ be  an real eigenvalue associated to an eigenvector  $v\in \Tt_{\calW_{\xpa}}$ \ie 
\begin{align*}
	-\wtilde{\calH}_{uu}^{-1}\wtilde{\calH}^2_{zu}v=\eta v &\iff \Ppa{\eta\widetilde{\calH}_{uu}+\widetilde{\calH}^2_{zu}}v=0.
\end{align*}
Let us observe that since  $\xpa$ is an active strict saddle point for the problem, then $\widetilde{\calH}_{uu}+\widetilde{\calH}^2_{zu}$ has a strict negative eigenvalue. Indeed, 
\[
\widetilde{\calH}_{uu}+\widetilde{\calH}^2_{zu}=\proj{\Tt_{\calW_{\xpa}}}\Ppa{\nabla^2\wtilde{G}(\xpa)+\sum_{i=1}^{n}\bar{\lambda}_{i}\nabla^2_{uu}W_i(\xpa)+ (1-a)^2\Ppa{\nabla^2F(\xpa)-\frac{1}{\gamma}\nabla^2\psi(\xpa)}}\proj{\Tt_{\calW_{\xpa}}}.
\]
Combining  with the  fact that $\widetilde{\calH}_{uu}$ is positive definite means that there exists $\eta>1$ such that the matrix $\eta\widetilde{\calH}_{uu}+\widetilde{\calH}_{zu}$ is singular. By construction, we get that $\eta>1$ is an eigenvalue of $\nabla \mathbf{T}$ associated to the eigenvector $\begin{bmatrix}v\\0\end{bmatrix}$ of $\mathbf{T}$.

\section{Riemannian Geometry and Partial Smoothness}\label{sec:rieman-tool}

\subsection{Riemannian geometry}

Let $\calM$ be a $C^2-$smooth embedded submanifold of $\bbR^n$ around a point $x\in\bbR^n.$ With some abuse of terminology, we will  state  $C^2-$manifold instead of $C^2-$smooth embedded submanifold of $\bbR^n$. The natural embedding of a submanifold $\calM$ into $\bbR^n$ permits to define a Riemannian structure and to introduce geodesics on $\calM,$ and we simply say $\calM$ is a Riemannian manifold. We denote respectively $\tgtManif{\calM}{x}$ and $\normSp{\calM}{x}$ the tangent and normal space of $\calM$ at $x\in\calM.$ 

\paragraph*{Exponential map} Geodesics generalize the concept of straight lines from linear spaces to manifolds. It is a smooth curve from an interval of $\bbR$ to $\calM$, with intrinsic acceleration normal everywhere to $\calM.$ Roughly speaking, it is locally the shortest path between two points on $\calM.$ Let denote by $\geod{t}{x}{h}$ the value at $t\in\bbR$ of the geodesic starting $\geod{0}{x}{h}=x\in\calM$ with velocity $\dotgeod{t}{x}{h}=\dfrac{d\geod{t}{x}{h}}{dt}=h\in\tgtManif{\calM}{x}$ (uniquely defined). 
It is  important to realize that for every $h\in\tgtManif{\calM}{x}$ there exists an interval $I$ around $0$  and a unique geodesic $\geod{t}{x}{h}:I\rightarrow\calM$ such that 
$\geod{0}{x}{h}=x$ and $\dotgeod{0}{x}{h}=t.$ The mapping $\Expv_{x}:\tgtManif{\calM}{x}\rightarrow\calM,\quad h\mapsto \Exp{x}{h}=\geod{1}{x}{h},$ is called the \textit{Exponential map.} Given $x,x'\in\calM,$ and a direction $h\in\tgtManif{\calM}{x}$ we are want such map to fulfill $\Exp{x}{h}=x'=\geod{1}{x}{h}.$

\paragraph*{Parallel translation} Given two points $x,x'\in\calM $ let $\tgtManif{\calM}{x},\tgtManif{\calM}{x'}$ be the corresponding tangent spaces. Define $\tau:\tgtManif{\calM}{x}\rightarrow\tgtManif{\calM}{x'}$, the 
\textit{parallel translation} along the unique geodesic joining $x$ to $x'$, which is an isomorphism and isometry with respect to the Riemannian metric.

\paragraph*{Riemannian gradient and Hessian} For a vector $v\in\normSp{\calM}{x},$ the Weingarten map of $\calM$ at $x$ is the operator $\Wgtmap{x}(.,v):\tgtManif{\calM}{x}\rightarrow\tgtManif{\calM}{x}$ defined by: 
$$
\Wgtmap{x}(h,v)=-\proj{\tgtManif{\calM}{x}}dV[h],
$$
where $V$ is any local extension of $v$ to a normal vector field on $\calM$. The definition does not depend of the choice of the extension $V$.  The  Weingarten map as defined above is a symmetric linear operator which is closely tied to the second fundamental form of $\calM$ (\cite[Definition 5.48]{boumal2023intromanifolds}). 
Let $g$ be a real-valued function which is $C^2$ along $\calM$ around $x$. The covariant gradient of $g$ at $x'\in\calM$ is the vector  denoted $\nabla_{\calM}g(x')\in\tgtManif{\calM}{x'}$ defined by:

$$
\pscal{\nabla_{\calM}g(x'),h}=\dfrac{d}{dt}g(\proj{\calM}(x'+th))|_{t=0},\quad \forall h\in\tgtManif{\calM}{x'}, 
$$

where $\proj{\calM}$ is the projection onto $\calM.$ The covariant Hessian of $g$ at $x'$ is the symmetric linear mapping $\nabla_{\calM}^2g(x')$ from $\tgtManif{\calM}{x'}$ to itself which is defined as

$$
\pscal{\nabla_{\calM}^2g(x')h;h}=\dfrac{d^2}{dt^2}g(\proj{\calM}(x'+th))|_{t=0},\quad \forall h\in\tgtManif{\calM}{x'}.
$$

This definition agrees with the definition using geodesics or connections. Now,  assume that $\calM$ is a Riemannian embedded submanifold of $\bbR^n$ and $g$ has a $C^2-$smooth restriction on $\calM.$ This can be characterization by the existence of a $C^2-$ smooth extension (representative) of $g$ \ie a $C^2$-smooth function $\widetilde{g}$ on $\bbR^n$ such that $\widetilde{g}$ agrees with $g$ on $\calM.$ Thus, the Riemannian gradient $\nabla_{\calM}g(x')$ is also given by
$$
\nabla_{\calM}g(x')=\proj{\tgtManif{\calM}{x'}}\nabla\widetilde{g}(x'),
$$
and $\forall h\in\tgtManif{\calM}{x'},$ the Riemannian Hessian reads
$$
	\begin{aligned}
		\nabla^2_{\calM}g(x')h&= \proj{\tgtManif{\calM}{x'}}d\Ppa{\nabla_{\calM}g(x')}[h]=\proj{\tgtManif{\calM}{x'}}d\Ppa{\nabla_{\calM}\widetilde{g}(x')}[h]\\
		&=\proj{\tgtManif{\calM}{x'}}\nabla^2\widetilde{g}(x')h+ \Wgtmap{x'}(h,\proj{\normSp{\calM}{x'}}\nabla\widetilde{g}(x')).
	\end{aligned}
$$
If $\calM$ is affine or linear subspace of $\bbR^n,$ then $\calM=x+\tgtManif{\calM}{x}$ and 
$\Wgtmap{x'}(h,\proj{\normSp{\calM}{x'}}\nabla\widetilde{g}(x'))=0$ and finally 
\begin{align*}
	\nabla^2_{\calM}g(x')=\proj{\tgtManif{\calM}{x'}}\nabla^2\widetilde{g}(x')\proj{\tgtManif{\calM}{x'}}.
\end{align*}
The next two lemmas will be instrumental when analyzing the local convergence behaviour of our inertial algorithm. We refer to \cite[Section~2.6]{liang_thesis_2016} for their proofs.  
\begin{lemma}\label{lem: Tgtxk-to-Tgtx} Let $x\in\calM$ and $\xk$ a sequence converging to $x$ in $\calM.$ Denote $\tau_k:\tgtManif{\calM}{\xk}\rightarrow\tgtManif{\calM}{\xk}$ be the parallel translation along the unique geodesic joining  $x$ to $\xk$. Then, for any bounded vector $u\in\bbR^n,$ we have:
	\begin{equation}
		\Ppa{\frac{1}{\tau_k}\proj{\tgtManif{\calM}{\xk}}-\proj{\tgtManif{\calM}{x}}}u=o\Ppa{\normm{u}}.
	\end{equation}
\end{lemma}
\begin{lemma}\label{lem: Riem-Tayl} Let $x,x'$ be two close points in $\calM,$ denote $\tau: \tgtManif{\calM}{x}\rightarrow \tgtManif{\calM}{x'}$ the parallel  translation along the unique geodesic joining $x$ too $x'.$ The Riemannian Taylor expansion  of $g \in C^2(\calM)$ around $x$ reads, 
	\begin{equation}
		\frac{1}{\tau}\nabla_{\calM}g(x')=\nabla_{\calM}g(x)+\nabla^2_{\calM}g(x)\proj{\tgtManif{\calM}{x}}(x'-x) +o(\normm{x'-x}).
	\end{equation}
\end{lemma}

	%
\subsection{Partial Smoothness}
Introduced by Lewis in \cite{lewis_active_2002}, ``Partial smoothness'' captures the characteristics of the geometry of nonsmooth functions. It axiomatizes the notion of active/identifiable submanifold or identifiable surfaces in \cite{wright_identifiable_1993}. A partly smooth function is smooth along the identifiable submanifold and sharp transversally to the manifold. Therefore, the behaviour of the function of its minimizers depends essentially on its restriction to this manifold, hence offering a powerful framework for algorithmic and sensitivity analysis theory. 

\begin{definition}[Partly smooth function]\label{def:part-smoo} A function $g\in\Gamma_0(\bbR^n)$ is $C^2-$partly smooth at a point $x$ relative to the set $\calM$ containing $x$, if $\partial g(x)\neq\emptyset $ and $\calM$ is an embedded $C^2-$smooth submanifold and there exists a neighborhood $\scrV_{x}$ of $x$ such that the following properties hold
	\begin{enumerate}[label=(\roman*)]
		\item \textbf{(Smoothness)}  the restriction $g_{|\calM}$ is a $C^2$ function in the neighborhood $\scrV_{x}$; 
		\item \textbf{(Sharpness)} The affine hull of $\partial g(x)$ is a translation of the space $\normal{\calM}(x)$, \ie 
		\begin{equation*}
			S_x \eqdef \LinHull(\partial g(x))=\normal{\calM}(x)\Leftrightarrow T_x \eqdef \orth{\LinHull(\partial g(x))}=\tgtManif{\calM}{x}. 
		\end{equation*} 
		\item \textbf{(Continuity)} The set-valued mapping $\partial g$ is continuous at $x$ relative to $\calM$. 
	\end{enumerate}
\end{definition}
Observe that $S_x = T_x^\perp$ by definition. Throughout the rest of the work, we denote the class of $C^2-$partly smooth function at $x$ relative to $\calM$  by $\PSF{x}{\calM}$.


Owing to the definition of partial smoothness, we have the following facts. 
\begin{fact}\label{fct:normal_sharp}\textbf{(Local normal sharpness)} If $g\in\PSF{x}{\calM}$, then for all point $x'\in\calM$ near $x$ we have $\tgtManif{\calM}{x'}=T_{x'}$. In particular when $\calM$ is affine or linear, then $T_{x'}=T_x.$ 
\end{fact}

\begin{fact}\label{fct:grad_psf} If $g\in \PSF{x}{\calM}$, then  for all $x'\in \calM$ near $x$ we have
	$$
	\nabla_{\calM}g(x')=\proj{T_{x'}}\Ppa{\partial g(x')},
	$$
	and this does not depend on the smooth representation of $g$ on $\calM.$ In turn, for all $h\in T_{x'},$
	\begin{align*}
		\nabla^2_{\calM}g(x')h=\proj{T_{x'}}\nabla^2\widetilde{g}(x')h+\Wgtmap{x'}(h,\proj{T^\bot_{x'}}\nabla\widetilde{g}(x')),
	\end{align*}
	where $\widetilde{g}$ is a smooth representative of $g$ on $\calM$ and $\Wgtmap{x'}(\cdot,\cdot)$ is the Weingarten map of $\calM$ at $x$.
\end{fact}

\end{appendices}

\begin{acknowledgments}
The author would like to thank Jalal Fadili for fruitful discussions  and the French National Research Agency (ANR) for funding the project FIRST (ANR-19-CE42-0009) 
\end{acknowledgments}

\bibliographystyle{plain}

\bibliography{BPG_PSF.bib}

\end{document}